\documentclass[10pt]{amsart}

\usepackage[dvips]{color}
\usepackage{amsmath}
\usepackage{amsthm,amssymb,amsfonts}
\usepackage{epsfig}

\theoremstyle{plain}
\newtheorem{thm}{Theorem}[section]
\newtheorem{lem}[thm]{Lemma}
\newtheorem{cor}[thm]{Corollary}
\newtheorem{prop}[thm]{Proposition}
\newtheorem{fact}[thm]{Fact}

\newtheorem*{thm*}{Theorem}

\theoremstyle{remark}
\newtheorem{rmk}[thm]{Remark}

\newtheorem*{case}{Case}
\newtheorem{claim}[thm]{Claim}

\theoremstyle{definition}
\newtheorem{defn}[thm]{Definition}

\newtheorem{notation}[thm]{Notation}

\numberwithin{equation}{section}

\DeclareMathOperator{\ex}{ex}
\DeclareMathOperator{\bary}{bar}

\DeclareMathOperator{\supp}{supp}
\DeclareMathOperator{\interior}{int}
\DeclareMathOperator{\Id}{Id}
\DeclareMathOperator{\NW}{NW}
\DeclareMathOperator{\diam}{diam}
\DeclareMathOperator{\dist}{dist}

\def\Z{\mathbb Z}
\def\D{\mathbb D}
\def\C{\mathbb C}
\def\R{\mathbb R}

\def\N{\mathbb N}

\def\cS{\mathcal S}

\def\al{\alpha}
\def\o{\omega}
\def\M{\mathcal M}

\def\cA{\mathcal A}

\def\P{\mathcal P}

\def\H{\mathcal H}

\def\0{\mathbf{0}}
\def\F{\mathcal F}

\def\Cl{\mathcal{C}}

\def\erg{\text{erg}}
\def\Bl{\mathcal{BL}}
\def\htop{\mathbf{h}_{\textnormal{top}}}
\def\hsex{h_{\textnormal{sex}}}
\def\hres{h_{\textnormal{res}}}
\def\har{\textnormal{har}}

\title{Orders of Accumulation of Entropy on Manifolds}
\author{Kevin McGoff}

\begin{document}

\maketitle

\begin{abstract}
For a continuous self-map $T$ of a compact metrizable space with finite topological entropy, the order of accumulation of entropy of $T$ is a countable ordinal that arises in the theory of entropy structure and symbolic extensions. Given any compact manifold $M$ and any countable ordinal $\al$, we construct a continuous, surjective self-map of $M$ having order of accumulation of entropy $\al$. If the dimension of $M$ is at least $2$, then the map can be chosen to be a homeomorphism.
\end{abstract}

\tableofcontents

%%%%%%%%%%%%%%%%%%%%%%%%%%%%%%%%%%%%%%%%%%%%%%%%%%%%%%%%%%%%%%%%%%%%%%%%%%%%%%%%%%%%%%%%%%%%%%%%%%%%%%%%%%%%%%%%5

\section{Introduction}

For the purposes of this work, a topological dynamical system consists of a pair $(X,T)$, where $X$ is a compact metrizable space and $T$ is a continuous surjection of $X$ to itself. For such a system $(X,T)$, the topological entropy $\htop(T)$ provides a well-studied measure of the topological dynamical complexity of the system. We only consider systems with $\htop(T) < \infty$. Let $M(X,T)$ be the space of Borel probability measures on $X$ that are invariant under $T$. The entropy function $h:M(X,T) \rightarrow [0,\infty)$, where $h(\mu)$ is the metric entropy of the measure $\mu$, quantifies the amount of complexity in the system that lies on generic points for each measure $\mu$ in $M(X,T)$. In this sense, the entropy function $h$ describes both \textit{where} and \textit{how much} complexity lies in the system. The theory of entropy structure developed by Downarowicz \cite{D} produces a master entropy invariant in the form of a distinguished class of sequences of functions on $M(X,T)$ whose limit is $h$. The entropy structure of a dynamical system completely determines almost all previously known entropy invariants (\textit{e.g.} the topological entropy, the entropy function on invariant measures, the tail entropy or topological conditional entropy \cite{M}, the symbolic extension entropy function) and, in fact, produces new invariants. Furthermore, the theory of entropy structure and symbolic extensions provides a rigorous description of \textit{how entropy emerges on refining scales}. Entropy structure has attracted interest in the dynamical systems literature \cite{A, Bur2, BurNew, DF, D, DM, DN}, especially with the intention of understanding the symbolic extensions of various classes of smooth dynamical systems.

The purpose of the current work is to investigate a new entropy invariant arising from the theory of entropy structure and symbolic extensions: the order of accumulation of entropy, which is a countable ordinal associated to the system $(X,T)$, denoted $\al_0(X,T)$ or just $\al_0(T)$. The order of accumulation of entropy of the system is an invariant of topological conjugacy that measures, roughly speaking, \textit{over how many distinct ``layers" residual entropy emerges} \cite{D}. It is shown in \cite{BM}, using a realization theorem of Downarowicz and Serafin \cite{D, DS}, that all countable ordinals appear as the order of accumulation for a minimal homeomorphism of the Cantor set. If follows from work of Buzzi \cite{Buzzi} that if $f$ is a $C^{\infty}$ self-map of a compact manifold, then $\al_0(f)=0$ (see Theorem 7.8 in \cite{BFF}). Our main result, which is contained in Theorem \ref{mainThm}, states that if $M$ is a compact manifold and $\al$ is a countable ordinal, then there exists a continuous surjection $f: M \rightarrow M$ such that $\al_0(f) = \al$. Furthermore, if $\dim(M) \geq 2$, then $f$ can be chosen to be a homeomorphism. The proof of this theorem gives a much more concrete construction of dynamical systems with prescribed order of accumulation than the proofs in \cite{BM}, which rely on a realization theorem of Downarowicz and Serafin.

The paper is organized as follows. In Section \ref{Rel2DynSys}, we recall the basic notions and facts in the theory of entropy structures and symbolic extensions. Section \ref{TowersAndPrincExts} contains some lemmas regarding the behavior of several entropy invariants under certain suspensions and extensions. The proof of the main result involves inductively ``blowing up'' periodic points and ``sewing in'' more complicated dynamical behavior. The operation of ``blowing up'' periodic points and ``sewing in'' more complicated dynamics is carried out in Section \ref{BlowAndSew}, where we need only work in dimensions $1$ and $2$. Section \ref{TwoLemmas} contains some technical lemmas in which the transfinite sequence is computed for some some specific instances of maps resulting from the blow-and-sew construction. The transfinite induction scheme is executed in Section \ref{InductionScheme}, and proofs of the main results are then given in Section \ref{mainResults}.

%%%%%%%%%%%%%%%%%%%%%%%%%%%%%%%%%%%%%%%%%%%%%%%%%%%%%%%%%%%%%%%%%%%%%%%%%%%%%%%%%%%%%%%%%%%%%%%%%%%%%%%%
\section{Background} \label{Rel2DynSys}

We assume some basic familiarity with ordinals (see, for instance, \cite{Pin}) and metrizable Choquet simplices (see \cite{Ph}), but in this section we present the definitions and facts required for the following sections. We will denote by $\N$ the set of positive integers.

\begin{defn}
In this work, a dynamical system consists of a pair $(X,T)$, where $X$ is a compact metrizable space and $T: X \to X$ is a continuous surjection.
\end{defn}
Furthermore, we assume that the topological entropy of $T$ is finite, $\htop(T) < \infty$. For references on the ergodic theory of such topological dynamical systems, see \cite{Pet,W}.  %We choose this context because the theory of entropy structures and symbolic extensions, as developed by Boyle and Downarowicz in \cite{BD,D}, works within this context.

\subsection{Choquet simplices and $M(X,T)$}

Let $K$ be a compact, convex subset of a locally convex topological vector space. Let $\M(K)$ be the space of all Borel probability measures on $K$ with the weak* topology. The barycenter map, $\bary : \M(K) \to K$, is defined as follows: for $\mu$ in $\M(K)$, let $\bary(\mu)$ be the unique point in $K$ such that for each continuous affine function $f : K \to \R$,
\begin{equation*}
f(\bary(\mu)) = \int_K f  \; d\mu.
\end{equation*}
The barycenter map itself is continuous and affine.
\begin{defn}[\cite{AE} p. 69]
Let $K$ be a metrizable, compact, convex subset of a locally convex topological vector space. Then $K$ is a metrizable \textbf{Choquet simplex} if the dual of the continuous affine functions on $K$ is a lattice.
\end{defn}
We only need Choquet's characterization of metrizable Choquet simplices (see \cite{Ph}): a metrizable, compact, convex subset of a locally convex topological vector space is a metrizable Choquet simplex if and only if for each point $x$ in $K$, there exists a unique measure $\P_x$ in $\M(K)$ such that $\P_x( K \setminus \ex(K)) = 0$ and $\bary(\P_x) = x$.

Suppose $K$ is a metrizable Choquet simplex. A Borel measurable function $f: K \to \R$ is called \textbf{harmonic} if, for each $x$ in $K$ and each $\mathcal{Q}$ in $\M(K)$ with $\bary(\mathcal{Q})=x$, we have
\begin{equation*}
f(x) = \int f \, d\mathcal{Q}.
\end{equation*}
Using that $\P_x$ is the unique measure supported on the extreme points of $K$ with barycenter $x$, one may check that $f$ is harmonic if and only if $f(x) = \int f \, d\P_x$ for each $x$ in $K$. If $f$ is a real-valued function defined on the extreme points of $K$, then we define the \textbf{harmonic extension} of $f$ to be the function $f^{\har} : K \to \R$ given for $x$ in $K$ by $f^{\har}(x) = \int f d\P_x$. We also define $f : K \to \R$ to be supharmonic if, for each $x$ in $K$ and each $\mathcal{Q}$ in $\M(K)$ such that $\bary(\mathcal{Q}) = x$, it holds that $f(x) \geq \int f d\mathcal{Q}$.

For a dynamical system $(X,T)$, we write $M(X,T)$ to denote the space of Borel probability measures on $X$ that are invariant under $T$. We give $M(X,T)$ the weak* topology. It is well known that in this setting $M(X,T)$ is a metrizable, compact, convex subset of a locally convex topological vector space (see, for example, \cite{Glas, Pet}). The extreme points of $M(X,T)$ are exactly the ergodic measures, $M_{\erg}(X,T)$. Also, the statement that each invariant measure $\mu$ in $M(X,T)$ has a unique ergodic decomposition \cite{Glas, Pet} implies that $M(X,T)$ is a metrizable Choquet simplex (using Choquet's characterization). In other words, we have that for each $\mu$ in $M(X,T)$, there exists a unique measure $\P_{\mu}$ in $\M(M(X,T))$ such that $\P_{\mu}(M(X,T) \setminus M_{\erg}(X,T)) = 0$ and $\bary(\P_{\mu}) = \mu$.

\subsection{Dynamical systems notations}
We need some notation.
\begin{notation} Let $(X,T)$ be a dynamical system.
\begin{itemize}
\item Let $A$ be a Borel measurable subset of $X$. We make the convention that $M(A,T) = \{\mu \in M(X,T) : \mu(X \setminus A) = 0\}$. %Note that if $A$ is compact and $T$-invariant, then $M(A,f) = M(A,f|_A)$.
\item Let $\NW(T)$ denote the non-wandering set for $(X,T)$.
\item A measure $\mu$ in $M(X,T)$ as totally ergodic if $\mu$ is ergodic for the system $(X,T^n)$, for all $n \in \N$.
\item If $\theta = \{x_0, \dots, x_{n-1}\}$ is a $T$-periodic orbit, then we let $\mu_{\theta}$ denote the periodic measure $\frac{1}{n} \sum_{k=0}^{n-1} \delta_{x_k}$, where $\delta_x$ is the point mass concentrated at the point $x$.
\item Let $h : M(X,T) \rightarrow [0, \medspace \infty)$ be the function that assigns to each measure in $M(X,T)$ its metric entropy with respect to the system $(X,T)$. When we wish to emphasize the dependence of $h$ on the system $(X,T)$, we write $h^T$. Also, if $\cA$ is a Borel partition of $X$, then we denote by $h^T(\mu,\cA)$ the entropy of the partition $\cA$ with respect to the measure-preserving system $(T,\mu)$.
\item If $\mu$ is a Borel probability measure on the space $X$, then $\supp(\mu)$ is the intersection of all the closed sets $C$ in $X$ such that $\mu(C)=1$.
\end{itemize}
\end{notation}
Recall that if $\mu$ is in $M(X,T)$, then $\supp(\mu) \subset \NW(T)$.

\begin{defn}[\cite{Bowen}] \label{hExpansiveDef}
Let $T$ be a continuous self-map of the compact metric space $X$. Let $\epsilon >0$, $x \in X$, and $\Phi_{\epsilon}(x) = \{ y \in X : d(T^n x,T^n y) \leq \epsilon \text{ for all } n\}$. If there exists $\epsilon > 0$ such that the entropy of $T$ on the set $\Phi_{\epsilon}(x)$ is $0$ for all $x \in X$, then $(X,T)$ is $h$-expansive.
\end{defn}

\subsection{Upper semi-continuity}

If $E$ is a compact metrizable space and $f : E \to \R$, then we denote by $||f||$ the supremum norm of $f$. For $x$ in $E$, we define
\begin{equation*}
\limsup_{y \to x} f(y) = \max \Bigl( f(x), \sup\{ \limsup_{n \to \infty} f(x_n) : \{x_n\}_n \subset E \setminus \{ x\}, \lim_n x_n = x \} \Bigr).
\end{equation*}
\begin{defn}
Let $E$ be a compact metrizable space, and let $f : E \rightarrow \R$. Then $f$ is \textbf{upper semi-continuous} (u.s.c.) if one of the following equivalent conditions holds for all $x$ in $E$,
\begin{enumerate}
 \item $f = \inf_{\al} g_{\al}$ for some family $\{g_{\al}\}_{\al}$ of continuous functions;
 \item $f = \lim_{n} g_n$ for some nonincreasing sequence $(g_n)_{n\in \N}$ of continuous functions;
 \item For each $r \in \R$, the set $\{x : f(x) \geq r \}$ is closed;
 \item $\limsup_{y\rightarrow x} f(y) \leq f(x)$, for all $x \in E$.
\end{enumerate}
For any $f : E \rightarrow \R$, the \textbf{upper semi-continuous envelope} of $f$, written $\widetilde{f}$, is defined by letting $\widetilde{f} \equiv \infty$ if $f$ is unbounded, and otherwise
\begin{equation*}
 \widetilde{f}(x) = \inf \{ g(x) : g \text{ is continuous, and } g \geq f \}, \text{ for all $x$ in $E$}.
\end{equation*}
\end{defn}
\noindent Note that $\widetilde{f}$ is the smallest u.s.c. function greater than $f$ and satisfies
\begin{equation*}
 \widetilde{f}(x) =  \limsup_{y \rightarrow x} f(y).
\end{equation*}
It is immediately seen that for any $f, g : E \rightarrow \R$, $\widetilde{f+g} \leq \widetilde{f}+\widetilde{g}$, with equality holding if $f$ or $g$ is continuous. We remark that if $f : E \to [0,\infty)$ is bounded and u.s.c., then $f$ achieves its supremum. Also, if $K$ is a Choquet simplex and $f: K \to \R$ is concave and u.s.c., then $f$ is supharmonic.

\subsection{Entropy structure and symbolic extensions}

\begin{defn} Let $M$ be a compact metrizable space. A \textbf{candidate sequence} on $M$ is a non-decreasing sequence $(h_k)$ of functions from $M$ to $[0,\infty)$ such that $\lim_k h_k$ exists and is bounded. We assume by convention that $h_0 \equiv 0$. Given two candidate sequences $\H = (h_k)$ and $\F = (f_k)$ defined on the same space, we say that $\H$ \textbf{uniformly dominates} $\F$, written $\H \geq \F$, if for each $\epsilon > 0$, and for each $k$, there exists $\ell$, such that $f_k \leq h_{\ell}+\epsilon$. The candidate sequences $\H$ and $\F$ are \textbf{uniformly equivalent}, written $\H \cong \F$, if $\H \geq \F$ and $\F \geq \H$. Note that uniform equivalence is, in fact, an equivalence relation.
\end{defn}
The uniform equivalence relation captures the manner in which sequences converge to their limit. For example, if two sequences converge uniformly to the same limit function, then they are uniformly equivalent. Also, if $(h_k)$ and $(f_k)$ are two candidate sequences on a compact metrizable space, then $\lim_k ||h_k-f_k|| = 0$ implies $(h_k) \cong (f_k)$, but $(h_k) \cong (f_k)$ does not necessarily imply $\lim_k ||h_k - f_k||=0$.

\begin{defn}[\cite{D}]
Let $X$ be a compact metrizable space and $T : X \to X$ a continuous surjection. For any continuous function $f : X \to [0,1]$, let $\cA_{f}$ be the partition of $X \times [0,1]$ consisting of the set $\{(x,t) : f(x) \geq t\}$ and its complement. If $\F = \{f_1, \dots, f_n\}$ is a finite collection of continuous functions $f_i : X \to [0,1]$, then let $\cA_{\F} = \vee_{i=1}^{n} \cA_{f_i}$. Let $\{\F_k\}_k$ be an increasing sequence of finite sets of continuous functions from $X$ to $[0,1]$ chosen so that the partitions $\cA_{\F_k}$ separate points (such sequences exist \cite{D}). Let $\lambda$ be Lebesgue measure on $[0,1]$. We define $\H^{\textnormal{fun}}(T) = (h_k)$ to be the candidate sequence on $M(X,T)$ given by $h_k(\mu) = h^{T \times \Id}(\mu \times \lambda,\cA_{\F_k})$.
\end{defn}
\begin{defn}[\cite{D}]
Let $(X,T)$ be a dynamical system. A candidate sequence $\H$ on $M(X,T)$ is an \textbf{entropy structure} for $(X,T)$ if $\H \cong \H^{\textnormal{fun}}(T)$. We may also refer to the entire uniform equivalence class of candidate sequences containing $\H^{\textnormal{fun}}(T)$ as \textit{the} entropy structure of $(X,T)$.
\end{defn}

Downarowicz showed that many of the known methods of computing or defining entropy can be adapted to become an entropy structure. For example, suppose $(X,T)$ is a dynamical system with a refining sequence $\{P_k\}_k$ of finite Borel partitions of $X$ such that the boundaries of all partition elements have zero measure for all $T$-invariant measures. Then the sequence of functions $(h_k)$ defined for $\mu$ in $M(X,T)$ by $h_k(\mu) = h^T(\mu,P_k)$ is an entropy structure for $(X,T)$. It may happen, though, that a particular system does not admit such a sequence of partitions (for example, if the system has an interval of fixed points). In such a case, we give another example of an entropy structure, known as the Katok entropy structure \cite{D}.
\begin{defn}[\cite{D}] \label{KatokESDef}
For an ergodic measure $\mu$ in $M(X,T)$, $\epsilon >0$ and $0< \sigma < 1$, let
\begin{equation*}
h(\mu,\epsilon,\sigma) = \limsup_n \frac{1}{n} \log \min \{ |E| : \mu\bigl( \cup_{x \in E} B(x,n,\epsilon) \bigr) > \sigma \},
\end{equation*}
where $B(x,n,\epsilon)$ is the $(n,\epsilon)$ Bowen ball about the point $x$. For an invariant but non-ergodic measure $\mu$, define $h(\mu,\epsilon,\sigma)$ by harmonic extension. Then for any sequence $\{\epsilon_k\}_k$ tending to $0$, the sequence of functions $h_k(\mu) = h(\mu,\epsilon_k,\sigma)$ is an entropy structure (for proof, see \cite{D} if $T$ is a homeomorphism and \cite{BurThesis} if $T$ is merely continuous).
\end{defn}

\begin{notation}
Let $\H = (h_k)$ be a candidate sequence on $K$, and let $\pi : L \to K$. We write $\H \circ \pi$ to denote the candidate sequence on $L$ given by $h_k \circ \pi$. Also, if $S$ is a subset of $K$, let $\H|_S$ be the candidate sequence on $S$ given by $(h_k|_S)$.
\end{notation}

\begin{defn}
Let $\H$ be a candidate sequence. The \textbf{transfinite sequence associated to} $\H$, which we write as $(u_{\al}^{\H})$, is defined by transfinite induction as follows. Let $\tau_k = h - h_k$. Then
\begin{itemize}
\item let $u^{\H}_0  \equiv 0$;
\item if $u_{\al}^{\H}$ has been defined, let $u_{\al+1}^{\H}  = \lim_{k} \widetilde{u_{\al}^{\H} + \tau_k}$;
\item if $u_{\beta}^{\H}$ has been defined for all $\beta < \al$, where $\al$ is limit ordinal, then let
\begin{equation*}
u_{\al}^{\H}  = \widetilde{\sup_{\beta < \al} u_{\beta}^{\H}}.
\end{equation*}
\end{itemize}
\end{defn}
The sequence $(u_{\al}^{\H})$ is non-decreasing in $\al$ and does not depend on the choice of representative of uniform equivalence class \cite{D}, which allows us to make the following definition.
\begin{defn}
Let $(X,T)$ be a topological dynamical system. Then \textbf{the transfinite sequence associated to} $(X,T)$ is the sequence $(u_{\gamma}^{\H(T)})$, where $\H(T)$ is an entropy structure for $T$.
\end{defn}

Note that for each $\al$, the function $u_{\al}^{\H}$ is either identically equal to $+\infty$ or it is u.s.c. into $\R$ (since a non-increasing limit of u.s.c. functions is u.s.c.). The sequence $(u_{\al}^{\H})$ is also sub-additive in the following sense.

\begin{prop}[\cite{BM}] \label{UsubAdd}
Let $\H$ be a candidate sequence on $E$. Then for any two ordinals $\al$ and $\beta$,
\begin{equation*}
 u_{\al+\beta}^{\H} \leq u_{\al}^{\H} + u_{\beta}^{\H}.
\end{equation*}
\end{prop}

If $\H$ is a candidate sequence, then by Theorem 3.3 in \cite{BD}, there exists a countable ordinal $\al$ such that $u_{\al}^{\H} = u_{\al + 1}^{\H}$, which then implies that $u_{\beta}^{\H} = u_{\al}^{\H}$ for all $\beta > \al$.
\begin{defn}
If $\H$ is a candidate sequence, then the least ordinal $\al$ such that $u_{\al}^{\H} = u_{\al + 1}^{\H}$ is called the \textbf{order of accumulation} of $\H$, which we write as $\al_0(\H)$. If $(X,T)$ is a topological dynamical system, then the \textbf{order of accumulation of entropy} of $(X,T)$, written $\al_0(X,T)$ or just $\al_0(T)$, is defined as $\al_0(\H(T))$, where $\H(T)$ is an entropy structure for $T$.
\end{defn}

To understand the meaning of the transfinite sequence and the order of accumulation of entropy of $(X,T)$, we turn to the connection between symbolic extensions and entropy structure.
\begin{defn}
Let $(X,T)$ be a topological dynamical system. A \textbf{symbolic extension} of $(X,T)$ is a subshift $(Y,S)$ of a (two-sided) full shift on a finite alphabet, along with a continuous surjection $\pi: Y \rightarrow X$ such that $T \circ \pi = \pi \circ S$.
\end{defn}

\begin{defn}
If $(Y,S)$ is a symbolic extension of $(X,T)$ with factor map $\pi$, then the \textbf{extension entropy function}, $h_{\textnormal{ext}}^{\pi} : M(X,T) \rightarrow [0,\infty)$, is defined for $\mu$ in $M(X,T)$ by
\begin{equation*}
h_{\textnormal{ext}}^{\pi}(\mu) = \sup \{ h(\nu) : \pi \mu = \nu \}.
\end{equation*}
The \textbf{symbolic extension entropy function} of a dynamical system $(X,T)$, $\hsex : M(X,T) \rightarrow [0, \medspace \infty]$, is defined for $\mu$ in $M(X,T)$ by
\begin{equation*}
\hsex(\mu) = \inf \{ h_{\textnormal{ext}}^{\pi}(\mu) : \pi \text{ is the factor map of a symbolic extension of } (X,T) \},
\end{equation*}
and the \textbf{residual entropy function}, $\hres: M(X,T) \rightarrow [0,\infty]$, is defined as
\begin{equation*}
\hres = \hsex-h.
\end{equation*}
If $(X,T)$ does not admit symbolic extensions, we let $\hsex \equiv \infty$ and $\hres \equiv \infty$, by convention.
\end{defn}
We think of a symbolic extension as a ``lossless finite encoding'' of the dynamical system $(X,T)$ \cite{D}. The symbolic extension entropy function quantifies at each measure the minimal amount of entropy that must be present in such an encoding.

The study of symbolic extensions  is related to entropy structures by the following remarkable result of Boyle and Downarowicz.
\begin{thm}[\cite{BD}] \label{SexEntThm}
Let $(X,T)$ be a dynamical system with entropy structure $\H$. Then
\begin{equation*}
\hsex = h + u_{\al_0(T)}^{\H}.
\end{equation*}
\end{thm}
Note that the conclusion of Theorem \ref{SexEntThm} could be restated as $\hres = u_{\al_0(\H)}^{\H}$. This theorem relates the notion of how entropy emerges on refining scales to the symbolic extensions of a system, showing that there is a deep connection between these topics. Using this connection, some progress has been made in understanding the symbolic extensions of certain classes of dynamical systems. For examples of these types of results, see \cite{A, BD, Bur2, BurNew, DF, D, DM, DN}. In light of Theorem \ref{SexEntThm}, we observe that the order of accumulation of entropy measures over how many ``layers'' residual entropy accumulates in the system. 

\subsection{Background lemmas}

The following lemma (Lemma \ref{EmbeddingLemma}), which is proved in \cite{BM}, will be used to compute the transfinite sequence associated to the systems that appear in Sections \ref{TowersAndPrincExts} - \ref{mainResults}. Although the entropy function $h$ is a harmonic function on the simplex of invariant measures, the functions $u_{\al}^{\H}$ are not harmonic in general. Lemma \ref{EmbeddingLemma} is useful because it nonetheless provides an integral representation of the functions $u_{\al}^{\H}$. A candidate sequence $(h_k)$ on a Choquet simplex such that each function $h_k$ is harmonic will be referred to as a harmonic candidate sequence. Let $K$ be a metrizable Choquet simplex with $E = \ex(K)$. In the following lemma we identify $\M(\overline{E})$ with the set $\{\mu \in \M(K) : \supp(\mu) \subset \overline{E}\}$ in the natural way, where $\overline{E}$ denotes the closure of $E$ in $K$. Also, if $f$ is a measurable function defined on the measurable subset $C$ of $K$ and $\mu$ is a measure on $K$, then $\int_C f \; d\mu$ is defined to be the integral with respect to $\mu$ of the function
\begin{equation*}
x \mapsto \left\{ \begin{array}{ll} f(x),  & \text{ if } x \in C \\
                                     0, & \text{ if } x \notin C.
                   \end{array}
           \right.
\end{equation*}
\begin{lem}[Embedding Lemma \cite{BM}]\label{EmbeddingLemma}
Let $K$ be a metrizable Choquet simplex with $E = \ex(K)$. Suppose $\H$ is a harmonic candidate sequence on $K$ and there is a set $F \subset E$ such that the sequence $\{(h-h_k)|_{E \setminus F} \}_k$ converges uniformly to zero. Let $C$ be a closed subset of $K$ such that $F \subset C$, and let $\Phi : \M(\overline{E}) \rightarrow K$ be the restriction of the barycenter map. Then for all ordinals $\al$ and for all $x$ in $K$,
\begin{equation}\label{embeddingEqn}
u_{\al}^{\H}(x) = \max_{\mu \in \Phi^{-1}(x)} \int_{C} u_{\al}^{\H|_C} \; d\mu,
\end{equation}
and $\al_0(\H) \leq \al_0(\H|_C)$. In particular, if $x$ is an extreme point of $K$ contained in $\C$, then $u_{\al}^{\H}(x) = u_{\al}^{\H|_C}(x)$ for all ordinals $\al$.
\end{lem}

We end this section by stating some facts that will be used repeatedly in the following sections. Facts \ref{TransSeqFacts} (1)-(4) are easily checked from the definitions, and Fact \ref{TransSeqFacts} (5), which is proved in \cite{BM}, follows from the fact that the u.s.c. envelope of a concave function is concave and the limit of concave functions is concave.
\begin{fact} \label{TransSeqFacts}  Let $M$, $M_1$, $M_2$, and $K$ be compact metrizable spaces. The for all ordinals $\gamma$, the following hold.
 \begin{enumerate}
\item If $\H$ is a candidate sequence on $M$ and $U$ is an open neighborhood of $x$ in $M$, then $u_{\gamma}^{\H}(x) = u_{\gamma}^{\H|_U}(x)$.
\item Suppose that $\H$ is a candidate sequence on $M_1$ and $M_2 \subset M_1$. Then $u_{\gamma}^{\H|_{M_2}} \leq u_{\gamma}^{\H|_{M_1}}$.
\item Suppose that $\pi : M \to K$ is a continuous surjection, $\F$ is a candidate sequence on $K$, and $\H = \F \circ \pi$. Then $u_{\gamma}^{\H} \leq u_{\gamma}^{\F} \circ \pi$.
\item Suppose $\pi : M \to K$ is continuous, surjective, and open (which is satisfied, in particular, if $\pi$ is a homeomorphism), $\F$ is a candidate sequence on $K$, and $\H = \F \circ \pi$. Then $u_{\gamma}^{\H} = u_{\gamma}^{\F} \circ \pi$.
\item Suppose $\H$ is a harmonic candidate sequence on a metrizable Choquet simplex $M$. Then $u_{\gamma}^{\H}$ is concave for all $\gamma$, and since $u_{\gamma}^{\H}$ is u.s.c., $u_{\gamma}^{\H}$ is also supharmonic. In particular, if $(X,T)$ is a topological dynamical system, then there exists a harmonic entropy structure $\H(T)$ for $T$ \cite{D}, and therefore $u_{\gamma}^{\H(T)}$ is concave and supharmonic for all $\gamma$.
 \end{enumerate}
\end{fact}

%%%%%%%%%%%%%%%%%%%%%%%%%%%%%%%%%%%%%%%%%%%%%%%%%%%%%%%%%%%%%%%%%%%%%%%%%%%%%%%%%%%%%%%%%%%%%%%%%%%%%%%%%%%%%%%%%%%%%%%
\section{Principal extensions and towers} \label{TowersAndPrincExts}

\begin{defn}
Let $(X,T)$ be a factor of $(Y,S)$ with factor map $\pi$. The system $(Y,S)$ is a \textbf{principal extension} of $(X,T)$ if $h^T(\pi \mu) = h^S(\mu)$ for all $\mu$ in $M(Y,S)$.
\end{defn}
If $(Y,S)$ is a principal extension of $(X,T)$, then we may refer to $S$ as a principal extension of $T$. The following fact is a basic result in the theory of entropy structures.
\begin{fact}[\cite{D}] \label{PrincExtFact}
If $S$ is a principal extension of $T$ with factor map $\pi$ and $\H(T)$ is an entropy structure for $T$, then $\H(T) \circ \pi$ is an entropy structure for $S$.
\end{fact}
\begin{lem} \label{PrincExtLemma}
Let $(X,f)$ be a topological dynamical system. Suppose there exists a compact subset $C$ of $M(X,f)$ such that for each ordinal $\gamma$ and each measure $\mu$ in $M(X,f)$,
\begin{equation} \label{HoffHypothesis}
 u_{\gamma}^{\H(f)}(\mu) = \int_{C} u_{\gamma}^{\H(f)|_{C}} \; d\P_{\mu}.
\end{equation}
Let $(Y,F)$ be a principal extension of $(X,f)$ with factor map $\pi$ and induced map $M(Y,F) \to M(X,f)$ also denoted by $\pi$. Suppose that $\pi|_{\pi^{-1}(C)}$ is a homeomorphism onto $C$. Then for each ordinal $\gamma$ and each measure $\nu$ in $M(Y,F)$,
\begin{equation*}
u_{\gamma}^{\H(F)}(\nu) = \int_{\pi^{-1}(C)} u_{\gamma}^{\H(F)|_{\pi^{-1}(C)}} \; d\P_{\nu} = u_{\gamma}^{\H(f)}(\pi(\nu)).
\end{equation*}
\end{lem}
\begin{proof}
Let $\H(f)$ be an entropy structure for $f$, and let $\H(F) = \H(f) \circ \pi$, which is an entropy structure for $F$ by Fact \ref{PrincExtFact}. By monotonicity (Fact \ref{TransSeqFacts} (2)), $u_{\gamma}^{\H(F)}(x) \geq u_{\gamma}^{\H(F)|_{\pi^{-1}(C)}}(x)$ for all $x$ in $\pi^{-1}(C)$. Since $\pi|_{\pi^{-1}(C)}$ is a homeomorphism onto $C$, $u_{\gamma}^{\H(F)|_{\pi^{-1}(C)}} = u_{\gamma}^{\H(f)|_{C}} \circ \pi$ (Fact \ref{TransSeqFacts} (4)). Combining these facts with Equation (\ref{HoffHypothesis}) and the fact that $u_{\gamma}^{\H(F)}$ is concave (Fact \ref{TransSeqFacts} (5)), we obtain that for all ordinals $\gamma$ and all $\nu$ in $M(Y,F)$,
\begin{align*}
u_{\gamma}^{\H(F)}(\nu) & \geq \int_{\pi^{-1}(C)} u_{\gamma}^{\H(F)|_{\pi^{-1}(C)}} \; d\P_{\nu} =  \int_{\pi^{-1}(C)} u_{\gamma}^{\H(f)|_{C}} \circ \pi \; d\P_{\nu} \\
 & = \int_{C} u_{\gamma}^{\H(f)|_{C}} \; d\P_{\pi\nu}  = u_{\gamma}^{\H(f)} ( \pi\nu).
\end{align*}
Since $\pi$ is continuous and surjective, $u_{\gamma}^{\H(F)} \leq u_{\gamma}^{\H(f)} \circ \pi$ (Fact \ref{TransSeqFacts} (3)). Combining the above inequalities, we obtain that $u_{\gamma}^{\H(F)} = u_{\gamma}^{\H(f)} \circ \pi$, and all of the above inequalities are equalities. This concludes the proof of the lemma.
\end{proof}

Now we turn our attention to simple towers. We begin with a definition.
\begin{defn}
Let $(X,T)$ be a topological dynamical system. Let $n$ and $p$ be natural numbers with $p \leq n$. Let $Y = X \times \{0, \dots, n-1\}$. We define a map $S : Y \rightarrow Y$ as follows. Let $S(x,i) = (x,i+1)$ for all $x$ in $X$ and $i \in \{0, \dots, n-1\}$. For each $x$ in $X$, let $S(x,n-1)=(T^p(x),0)$. We will refer to $(Y,S)$ (or possibly just $S$) as an $(n,p)$ \textbf{tower over} $(X,T)$ (or possibly just $T$).
\end{defn}
\begin{notation} \label{TowerNotation}
Suppose $(Y,S)$ is an $(n,p)$ tower over $(X,T)$. Let $Y_0 = X \times \{0\}$, and note that $Y_0$ is invariant under $S^n$. Let $\pi_1 : M(Y,S) \rightarrow M(Y_0,S^n|_{Y_0})$ be the map given by $\mu \mapsto \mu|_{Y_0}$. Let $\pi_2 : Y_0 \to X$ be projection onto $X$. With $\pi_2$ as the factor map, $(Y,S^n|_{Y_0})$ is a principal extension over $(X,T^p)$. Note that the maps $\pi_1$ and $\pi_2$ on measures are affine homeomorphisms. Further, recall that if $\mu$ is in $M(Y,S)$, then the measure-preserving systems $(S,\mu)$ and $(T^p,\pi_2 \circ \pi_1 (\mu))$ are measure-theoretically isomorphic. Let $\pi_3: M(Y,T^p) \rightarrow M(X,T)$ be the map $\pi_3(\mu) = \frac{1}{p} \sum_{i=0}^{p-1} T^i \mu$.
\end{notation}
\begin{defn} \label{TowerMapDef}
If $S$ is a tower over $T$ with notation as above, then the map $\psi = \pi_3 \circ \pi_2 \circ \pi_1$ will be referred to as \textbf{the map associated to the tower} $S$ \textbf{over} $T$.
\end{defn}

\begin{lem} \label{PowerESLemma}
If $(X,T)$ is a topological dynamical system with entropy structure $\H(T)$, then $p \H(T) \circ \pi_3$ is an entropy structure for $T^p$, where $\pi_3 : M(X,T^p) \to M(X,T)$ is defined by $\pi_3(\mu) = \frac{1}{p} \sum_{i=0}^{p-1} T^i \mu$.
\end{lem}
\begin{proof}
It is shown in \cite{BD} that every finite entropy dynamical system has a zero-dimensional principal extension. (In fact, \cite{BD} deals only with homeomorphisms, but the natural extension $\overline{T}$ of a continuous surjection $T$ is a homeomorphism and a principal extension of $T$, and then applying \cite{BD} to $\overline{T}$ yields a zero-dimensional principal extension of $T$.) Applying this fact to $(X,T)$, we fix a zero-dimensional principal extension $(X',T')$ of $(X,T)$ with factor map $\pi$. Then $(X',(T')^p)$ is a zero-dimensional principal extension of $(X,T^p)$ with factor map $\pi$. We let $\pi_3$ denote the averaging map from $M(X',(T')^p)$ to $M(X',T')$ as well as the averaging map from $M(X,T^p)$ to $M(X,T)$. Note that $\pi \circ \pi_3 = \pi_3 \circ \pi$. Now let $\H(T)$ be an entropy structure for $T$ and let $\H(T^p)$ be an entropy structure for $T^p$. We prove the lemma by showing that $\H(T^p)$ is uniformly equivalent to $p \H(T) \circ \pi_3$. Since $(X',T')$ is zero dimensional, there exists a refining sequence $\{P_k\}_k$ of clopen partitions of $X'$ with diameters tending to $0$. Let $\H(T') = (h^{T'}_k)$ and $\H((T')^p)=(h^{(T')^p}_k)$ be the entropy structures (for $T'$ and $(T')^p$ respectively) defined by this sequence of partitions, \textit{i.e.} $h^{T'}_k(\mu) = h^{T'}(\mu,P_k)$ and $h^{(T')^p}_k(\mu) = h^{(T')^p}(\mu,P_k)$. Then for any $\mu$ in $M(X',(T')^p)$, we have that $h^{(T')^p}_k(\mu) = h^{(T')^p}(\mu,P_k) = p h^{T'}(\pi_3(\mu),P_k) = p h^{T'}_k(\pi_3(\mu))$. Thus $\H((T')^p)$ is uniformly equivalent to $p \H(T') \circ \pi_3$. Since $T'$ is a principal extension of $T$ and $(T')^p$ is a principal extension of $T^p$, both with factor map $\pi$, we have that $\H(T') \cong \H(T) \circ \pi$ and $\H((T')^p) \cong \H(T^p) \circ \pi$. Combining these facts, we obtain that $p \H(T) \circ \pi \circ \pi_3 \cong \H(T^p) \circ \pi$. Since $\pi \circ \pi_3 = \pi_3 \circ \pi$, we see that $p \H(T) \circ \pi_3 \circ \pi \cong \H(T^p) \circ \pi$. Using this fact and the definition of uniform equivalence, we see that $p \H(T) \circ \pi_3 \cong \H(T^p)$, which finishes the proof of the lemma.
\end{proof}

\begin{lem} \label{TowerLemmaOne}
If $S$ is an $(n,p)$ tower over $T$ with associated map $\psi$ and $\H(T)$ is an entropy structure for $T$, then $\frac{p}{n} \H(T) \circ \psi$ is an entropy structure for $S$.
\end{lem}
\begin{proof}
We use Notation \ref{TowerNotation}. Note that the maps $\pi_1, \pi_2$ and $\pi_3$ are each continuous and affine. For any entropy structure $\H(S^n)$ of $S^n$, we have that $\frac{1}{n} \H(S^n) \circ \pi_1$ is an entropy structure for $S$ (Theorem 5.0.3 (3) in \cite{D}). If $\H(T^p)$ is an entropy structure for $T^p$, then $\H(T^p) \circ \pi_2$ is an entropy structure for $S^n$ by Fact \ref{PrincExtFact}, since $S^n$ is a principal extension of $T^p$ with factor map $\pi_2$. By Lemma \ref{PowerESLemma}, we have that if $\H(T)$ is an entropy structure for $T$, then $p \H(T) \circ \pi_3$ is an entropy structure for $T^p$. Combining these facts, we obtain that if $\H(T)$ is an entropy structure for $T$, then $\frac{p}{n} \H(T) \circ \psi$ is an entropy structure for $S$.
\end{proof}

\begin{lem} \label{TowerLemmaThree}
Let $(Y,S)$ be an $(n,p)$ tower over $(X,T)$ with associated map $\psi : M(Y,S) \to M(X,T)$. Let $\{\theta_m\}_m$ be a sequence of periodic orbits of $T$. Let $C(T) = \cap_{n=1}^{\infty} \overline{ \cup_{m \geq n} \{\mu_{\theta_m}\} } $. Suppose that each measure in $C(T)$ is totally ergodic. Further suppose that for all $\mu$ in $M(X,T)$ and all ordinals $\al$,
\begin{equation} \label{TowerLemmaThreeHyp}
u_{\al}^{\H(T)}(\mu) = \int_{C(T)} u_{\al}^{\H(T)|_{C(T)}} \; d\P_{\mu}.
\end{equation}
Let $\{\Theta_{\ell}\}$ be an enumeration of the $S$-periodic orbits in $\cup_m \theta_m \times \{ 0, \dots, n-1\}$, and let $C(S) = \cap_{n=1} \overline{ \cup_{\ell \geq n} \{ \mu_{\Theta_{\ell}} \} }$. Then
\begin{enumerate}
\item each $\nu$ in $C(S)$ is totally ergodic;
\item $\psi$ maps $C(S)$ homeomorphically onto $C(T)$;
\item for all $\nu$ in $M(Y,S)$ and all ordinals $\al$
\begin{equation*}
u_{\al}^{\H(S)}(\nu) = \int_{C(S)} u_{\al}^{\H(S)|_{C(S)}} \; d\P_{\nu} = \frac{p}{n} u_{\al}^{\H(T)}(\psi (\nu)).
\end{equation*}
\end{enumerate}
\end{lem}
\begin{proof} We use Notation \ref{TowerNotation}. Let $\mu$ be in $C(T)$. Since $\mu$ is invariant for $T$, it is also invariant for $T^p$ and we have $\pi(\mu)=\mu$. Further, $\mu$ is totally ergodic for $T$ by hypothesis, and therefore $\mu$ is totally ergodic for $T^p$. Hence $\mu$ is an extreme point in $M(X,T^p)$. If there were any other measure $\nu$ in $\pi_3^{-1}(\mu)$, then we would have $\mu = \frac{1}{p} \sum_{k=0}^{p-1} T^k \nu$, and thus $\mu$ would be a non-trivial convex combination of measures in $M(X,T^p)$, which would be a contradiction. Hence $\pi_3^{-1}(\mu) = \{\mu\}$. Since $\pi_2 \circ \pi_1: M(Y,S) \to M(X,T^p)$ is a homeomorphism, $(\pi_2 \circ \pi_1)^{-1}(\mu)$ consists of exactly one measure $\nu$. Since $(S,\nu)$ is measure-theoretically isomorphic to $(T^p,\mu)$ and $\mu$ is totally ergodic with respect to $T^p$, we have that $\nu$ is totally ergodic with respect to $S$. Combining these facts, we obtain that $\psi^{-1}(\mu)$ consists of exactly one measure, which is totally ergodic for $S$.

The fact that $\psi^{-1}(\mu)$ consists of exactly one measure for each $\mu$ in $C(T)$ implies that $\psi^{-1}(C(T)) = C(S)$ and that $\psi$ maps $C(S)$ bijectively onto $C(T)$. Since $C(S)$ is compact and $\psi$ is continuous, we conclude that $\psi$ maps $C(S)$ homeomorphically onto $C(T)$, which proves (2). The fact that $\psi^{-1}(\mu)$ is totally ergodic for $S$ implies that each $\nu$ in $C(S)$ is totally ergodic for $S$, proving (1).

Now Lemma \ref{TowerLemmaOne} implies that if $\H(T)$ is an entropy structure for $T$, then $\frac{p}{n}\H(T)\circ \psi$ is an entropy structure for $S$. Since $\psi|_{C(S)}$ is a homeomorphism onto $C(T)$, Fact \ref{TransSeqFacts} (4) implies that $u_{\gamma}^{\H(S)|_{C(S)}} = \frac{p}{n} u_{\gamma}^{\H(T)|_{C(T)}} \circ \psi|_{C(S)}$. Using this fact, as well as Equation (\ref{TowerLemmaThreeHyp}) and Facts \ref{TransSeqFacts} (2), (3), and (5), we obtain that for any $\nu$ in $M(Y,S)$,
\begin{align*}
\frac{p}{n} \int_{C(T)} u_{\gamma}^{\H(T)|_{C(T)}} \; d \P_{\psi(\nu)}
& =  \int_{C(S)} u_{\gamma}^{\H(S)|_{C(S)}} \; d \P_{\nu}
\leq u_{\gamma}^{\H(S)}(\nu)  \\
& \leq \frac{p}{n} u_{\gamma}^{\H(T)} (\psi(\nu))
 = \frac{p}{n} \int_{C(T)} u_{\gamma}^{\H(T)|_{C(T)}} \; d \P_{\psi(\nu)}.
\end{align*}
Thus the above inequalities are all equalities and the proof is complete.
\end{proof}

%%%%%%%%%%%%%%%%%%%%%%%%%%%%%%%%%%%%%%%%%%%%%%%%%%%%%%%%%%%%%%%%%%%%%%%%%%%%%%%%%%%%%%%%%%%%%%%%%%%%%%%%
\section{``Blow-and-sew''} \label{BlowAndSew}

We now begin building towards the proof of Theorem \ref{mainThm}. The main idea of the proof for dimensions $1$ and $2$ is that we may ``blow-up'' periodic points (to intervals in dimension $1$ and to discs in dimension $2$) and ``sew in'' more complicated dynamical behavior, and in the process we increase the order of accumulation in a controlled way. By iterating this procedure in a transfinite induction scheme, we obtain a proof of Theorem \ref{mainThm} in dimensions $1$ and $2$. The maps constructed in dimension $2$ can then be used to build maps in higher dimensions. In this section we describe and analyze the operation of ``blowing up'' a sequence of periodic orbits and ``sewing in'' other maps. The basic idea of this construction appears in Appendix C of \cite{BFF}. In this section, we assume $d \in \{ 1, 2\}$.

\begin{notation} \label{RnNotation}
Let $\D$ be the closed unit disc in $\R^d$. For a subset $S$ of $\R^d$, let $\interior(S)$ denote the interior of $S$, and let $\partial S$ be the boundary of $S$. For $r>0$ and $p$ in $\R^d$, we let $B(p, r)$ be the open ball of radius $r$ centered at $p$. Given $s > 0$ and a point $p$ in $\R^d$, let $A_{s,p}$ be the affine map of $\R^d$ given by $A_{s,p}(x) = sx + p$.
\end{notation}

We consider maps in the following class.
\begin{defn}
Define $\Cl_d$ to be the class of functions $f : \D \rightarrow \D$ with the following properties:
\begin{enumerate}
 \item $f$ is a continuous surjection, and if $d = 2$, then $f$ is a homeomorphism;
 \item $f|_{\partial \D} = \Id$;
 \item $\htop(f) < \infty$.
\end{enumerate}
\end{defn}

\begin{defn} \label{readyForOpDef}
Let $f : \D \rightarrow \D$ be continuous. Let $\{\theta_m\}_m$ be a sequence of periodic orbits for $f$, and let $S = \cup_m \theta_m$. We say that $f$ is \textbf{ready for operation on} $S$ if the following conditions are satisfied, where $Q = \cup_{k \geq 0} f^{-k}(S)$:
\begin{enumerate}
 \item for any $\nu$ in $\cap_{n = 1}^{\infty} \overline{\cup_{m \geq n} \{\mu_{\theta_m}\}}$, it holds that $\nu(\cup_m \theta_m) = 0$;
 \item the set $Q$ is countable and $Q \subset \interior(\D)$;
 \item for each point $x$ in $Q$, the derivative $Df_x$ is invertible, and if $d = 2$, then $\det Df_x > 0$ for each point $x$ in $Q$.
\end{enumerate}
\end{defn}
We remark that if $d = 2$, then in the above notation we have $Q = S$. To get non-zero orders of accumulation of entropy in dimension $1$, we must look outside the class of homeomorphisms because a homeomorphism of the circle or the unit interval has zero entropy, and therefore its order of accumulation of entropy is zero.

\subsection{The ``blow-and-sew" construction}

The following proposition carries out the ``blow and sew'' procedure. In the notation of the proposition, we think of the periodic orbits $\theta_m$ as being ``blown-up'' into disjoint discs, with towers over the maps $\chi_m$ being ``sewn in'' on the interiors of these discs.

\begin{prop} \label{BlowAndSewProp}
Suppose:
\begin{itemize}
 \item $f$ is a function in $\Cl_d$;
 \item $\{\theta_m\}_m$ is a sequence of periodic orbits for $f$, and $f$ is ready for operation on $\cup_m \theta_m$;
 \item $\{\chi_m\}_{m \in \N}$ is a sequence of functions in $\Cl_d$;
 \item for each natural number $m$, the sequence $\{\theta_{\ell}^m\}_{\ell}$ is a sequence of periodic orbits for $\chi_m$, and $\chi_m$ is ready for operation on $\cup_{\ell} \, \theta_{\ell}^m$;
 \item $\{\xi_m\}_{m \in \N}$ is a sequence of natural numbers satisfying $1 \leq \xi_m \leq |\theta_m|$ for each $m$ in $\N$.
 \item $\sup_m \; \frac{\xi_m}{|\theta_m|} \htop(\chi_m) < \infty$
\end{itemize}
Let $Q = \cup_{m,k} f^{-k}(\theta_m)$, and let $\{q_k\}_{k \in \N}$ be an enumeration of $Q$. Then there exist functions $F : \D \to \D$ and $\pi : \D \rightarrow \D$, a sequence $\{K_i\}_{i=0}^{\infty}$ of pairwise disjoint, compact subsets of $\D$, and a sequence $\{\phi_m\}_{m \in \N}$ of $C^{\infty}$ diffeomorphisms, $\phi_m : \D \times \{0, \dots, |\theta_m|-1\} \rightarrow K_m$, such that the following hold, where $L_k = \pi^{-1}(\{q_k\})$ for each $k$:
\begin{enumerate}
 \item $F$ is in $\Cl_d$;
 \item $\pi$ is a factor map from $(\D,F)$ to $(\D,f)$;
 \item $\pi(K_m)=\theta_m$, $L_k$ is $C^{\infty}$ diffeomorphic to $\D$ for each $k$, $\pi|_{\D \setminus (\cup_k L_k)}$ is injective, and $K_m \subset \cup_{q_k \in \theta_m} \interior(L_k)$ for each $m$ in $\N$.
 \item $K_i$ is $F$-invariant for each $i$, $K_0 = \D \setminus (\cup_{k=1}^{\infty} \interior(L_k))$, and $\cup_k \partial L_k$ is $F$-invariant.
 \item $\NW(F) \subseteq  \bigcup_{i=0}^{\infty} K_i$;
 \item $F|_{K_0}$ is a principal extension of $f$ with factor map $\pi|_{K_0}$, and for $\nu$ in $M(K_0 \setminus \cup_k \partial L_k,F)$, the map $\pi$ is a measure theoretic isomorphism between the measure preserving systems $(F,\nu)$ and $(f,\pi(\nu))$.
 \item $\phi_m$ is a topological conjugacy between $F|_{K_m}$ and a $(|\theta_m|,\xi_m)$ tower over $\chi_m$, for each $m$ in $\N$.
 \item $\cap_{n=1}^{\infty} \overline{ \cup_{m \geq n} M(K_m,F) } = (\pi)^{-1}(\cap_{n=1}^{\infty} \overline{ \cup_{m \geq n} \{\mu_{\theta_m}\} } ) \subset M(K_0 \setminus \cup_k \partial L_k,F)$, and $\pi$ maps $\cap_{n=1}^{\infty} \overline{ \cup_{m \geq n} M(K_m,F) }$ homeomorphically onto $\cap_{n=1}^{\infty} \overline{ \cup_{m \geq n} \{\mu_{\theta_m}\} }$;
 \item $F$ is ready for operation on $\cup_{m, \ell} \; \phi_m(\theta_{\ell}^m \times \{0, \dots, |\theta_m|-1\})$.
\end{enumerate}
\end{prop}

\begin{proof}
Let $f$, $\{\theta_m\}_{m \in \N}$, $\{\chi_m\}_{m \in \N}$, $\{\theta_{\ell}^m\}_{m,\ell \in \N}$, and $\{ \xi_m\}_{m \in \N}$ be given as in the hypotheses. Let $Q = \cup_{m, k \geq 0} f^{-k}(\theta_m)$, and let $Q = \{q_k\}_{k \in \N}$ be a enumeration of $Q$. The following lemma blows up each of the points in $Q$ into a disc.

\begin{lem} \label{BlowUpLemma}
Let $Q$ be a countable set contained in the interior of $\D$. Then there exists a summable sequence $\{\epsilon_k\}_{k \in \N}$ of positive real numbers, a sequence $\{p_k\}_{k \in \N}$ such that $\overline{B(p_k,\epsilon_k)}$ is contained in $\interior(\D)$ for each $k$ in $\N$, and a function $\pi : \D \rightarrow \D$ such that
\begin{enumerate}
  %\item $\phi$ is a $C^{\infty}$ orientation-preserving diffeomorphism;
  \item $\pi$ is continuous and surjective;
  \item $\pi^{-1}(\{q_k\}) = \overline{B(p_k, \epsilon_k)}$ for each $k$;
  \item $\pi|_{\D \setminus \cup_k \overline{B(p_k, \epsilon_k)}}$ is a homeomorphism onto its image, $\D \setminus Q$.
\end{enumerate}
\end{lem}
%\textcolor[rgb]{1,0,0}{Proof necessary?}
\begin{proof}
Let $Q = \{q_k\}$ be as in the hypotheses. Consider $\R^d \setminus \{0\}$ in polar coordinates: $(r,\theta) \in (0,\infty) \times S^{d-1}$. For $n$ in $\N$ and $\epsilon >0$, consider the function $R_{\epsilon,n} : \R^d \to \R^d$ given by $R_{\epsilon,n}(0) = 0$ and for $(r,\theta)$ in $\R^d \setminus \{0\}$,
\begin{equation*}
R_{\epsilon,n}(r,\theta) = \left\{ \begin{array}{ll}
                                      0, & \text{ if } r \leq \frac{\epsilon}{n} \\
                                      (\frac{n}{n-1}(r-\frac{\epsilon}{n}), \theta), & \text{ if } \frac{\epsilon}{n} \leq r \leq \epsilon \\
                                      (r,\theta), & \text{ otherwise.}
                                 \end{array}
                         \right.
\end{equation*}
Let $S_{\epsilon,n} : \R^d \to \R^d$ be given by $S_{\epsilon,n}(0)=0$ and $S_{\epsilon,n}(x) = R^{-1}_{\epsilon,n}(x)$ for $x \neq 0$. Now for $p$ in $\R^d$, let $R_{\epsilon,n,p} : \R^d \to \R^d$ be defined by $R_{\epsilon,n,p}(x) = R_{\epsilon,n}(x-p) +p$. Also define $S_{\epsilon,n,p} : \R^d \to \R^d$ to be $S_{\epsilon,n,p}(x) = S_{\epsilon,n}(x-p)+p$. Note that $R_{\epsilon,n,p}$ is continuous on $\R^d$ and $S_{\epsilon,n,p}$ is continuous on $\R^d \setminus \{p\}$. Also, $S_{\epsilon,n,p}|_{\R^d \setminus \{p\}}$ is a homeomorphism onto its image, with inverse given by $R_{\epsilon,n,p}|_{\R^d \setminus \overline{B(0,\frac{1}{n}\epsilon)}}$. Moreover, we have
\renewcommand{\labelenumi}{(\roman{enumi})}
\begin{enumerate}
 \item $d(R_{\epsilon,n,p}(x), R_{\epsilon,n,p}(y)) \leq \frac{n}{n-1} d(x,y)$;
% \item $\frac{n}{n-1} d(x,y) \geq d(L_{\epsilon,n,p}(x), L_{\epsilon,n,p}(y)) \leq d(x,y)+2\epsilon$;
 \item $d(x,R_{\epsilon,n,p}(x)) \leq \epsilon$;
 \item $d(x,S_{\epsilon,n,p}(x)) \leq \epsilon$.
\end{enumerate}

Choose a sequence $\{n_k\}_k$ of natural numbers such that $\prod_{k=1}^{\infty} \frac{n_k-1}{n_k} > 0$. Let $C = \prod_{k=1}^{\infty} \frac{n_k}{n_k-1} < \infty$. We make the following recursive definitions. Let $\delta_1 >0$ be such that $\dist(q_1,\partial \D) > \delta_1$. Let $f_1 = S_{\delta_1,n_1,q_1}$ and $g_1 = R_{\delta_1,n_1,q_1}$. If $\delta_k$, $f_k$ and $g_k$ are defined, choose $\delta_{k+1} >0$ such that $\delta_{k+1} < \dist(f_k(q_{k+1}), \partial \D \cup g_k^{-1}(\{q_1,\dots,q_k\})) $ and let $f_{k+1} = S_{\delta_{k+1},n_{k+1},f_k(q_{k+1})} \circ f_k$ and $g_{k+1} = g_k \circ R_{\delta_{k+1},n_{k+1},f_k(q_{k+1})}$. We also require that $\{\delta_k\}_k$ is summable.

The properties (i)-(iii) above imply that for any $k_1 \leq k_2$
\renewcommand{\labelenumi}{(\alph{enumi})}
\begin{enumerate}
 \item $d(g_{k_1}(x),g_{k_1}(y)) \leq (\prod_{k=1}^{k_1} \frac{n_k}{n_k-1} ) d(x,y)$;
% \item $d(f_K(x),f_K(y)) \geq (\prod_{k=1}^K \frac{n_k-1}{n_k} ) d(x,y)$;
 \item $d(g_{k_1}(x),g_{k_2}(x)) \leq \sum_{k=k_1}^{k_2} \delta_k$;
 \item $d(f_{k_1}(x),f_{k_2}(x)) \leq \sum_{k=k_1}^{k_2} \delta_k$.
\end{enumerate}
For each $k$, $f_k$ is continuous on $\R^d \setminus \{q_1, \dots, q_k\}$ and $g_k$ is continuous on $\R^d$. In fact, $f_k$ is a homeomorphism from $\R^d \setminus \{q_1, \dots, q_k\}$ to its image, and $g_k$ is its inverse. Note that the sequences $\{f_k\}_k$ and $\{g_k\}_k$ are uniformly Cauchy by properties (b) and (c) above. Therefore the pointwise limits $f(x) = \lim_k f_k(x)$ and $g(x) = \lim_k g_k(x)$ exist for all $x$ in $\R^d$. Since $f_k$ is continuous on $\R^d \setminus Q$ for all $k$, and since $\{f_k\}_k$ is uniformly Cauchy, $f$ is continuous on $\R^d \setminus Q$. The fact that $g_k$ is continuous on $\R^d$ for each $k$ and the sequence $\{g_k\}_k$ is uniformly Cauchy implies that $g$ is continuous. Using the fact that $\delta_{k+1} < \dist\bigl(f_k(q_{k+1}), \partial \D \cup g_k^{-1}(\{q_1,\dots,q_k\})\bigr)$ for each $k$, we observe that $f|_{\partial \D} = g|_{\partial \D} = \Id$ and if $x$ is in $g_m^{-1}(\{q_k\})$ where $k \leq m$, then $g_n(x) = g_m(x)$ for all $n \geq m$ and therefore $g(x) = g_m(x)$. This last observation means that if $g_m(x)$ is in $Q$ for any $m$, then $g(x)$ is in $Q$. We now consider $f$ and $g$ restricted to $\D$, and note that $f$ and $g$ act by the identity map on $\partial \D$. Also, each $g_k$ defines a continuous surjection and therefore $g$ does as well.

Let us check that for $x$ in $\D \setminus Q$, $g(f(x)) = x$. Note that $d(g_k(f_k(x)),g(f_k(x))) \leq \sum_{j=k}^{\infty} \delta_j$. Letting $k$ tend to infinity and using the continuity of $g$ gives that $g(f(x))=x$.

Finally we check that for $x$ in $g^{-1}(\D \setminus Q)$, $f(g(x))=x$. Let $x$ be in $g^{-1}(\D \setminus Q)$. Let $\epsilon >0$. Choose $K$ so large that $\sum_{k \geq K} \delta_k < \epsilon/3$. Since $g(x)$ is not in $Q$, $f_K$ is continuous at $g(x)$. Since $f_K$ is continuous at $g(x)$, there exists $\delta >0$ such that $d(y,g(x)) < \delta$ implies $d(f_K(y),f_K(g(x))) < \epsilon/3$. Then choose $M \geq K$ such that $d(g_M(x),g(x))<\delta$. Then
\begin{align*}
d(x,f(g(x)) &  \leq d(f_M(g_M(x)), f_K(g_M(x))) + \\ & \hspace{10 mm} + d(f_K(g_M(x)),f_K(g(x))) + d( f_K(g(x)), f(g(x))) \\ & \leq  \epsilon/3 + \epsilon/3 + \epsilon/3 = \epsilon.
\end{align*}
Since $\epsilon >0$ was arbitrary, we have that $x=f(g(x))$.

Now let $\epsilon_k = \frac{\delta_k}{n_k}$, $p_k = f(q_k)$, and $\pi = g$. Note that the conclusions of the lemma are satisfied by these choices.
\end{proof}

\subsubsection{Setup}
Now we proceed with the proof of Proposition \ref{BlowAndSewProp}. Choose $\{\epsilon_k\}_{k \in \N}$, $\{p_k\}_{k \in \N}$, and $\pi$ satisfying the assumptions and conclusions of Lemma \ref{BlowUpLemma}. These objects will remain fixed throughout the rest of the proof.

For the sake of notation, let $L_k = \overline{B(p_k, \epsilon_k)} $ and $L = \cup_k L_k$. Note that $\interior(L) = \cup_k B(p_k,\epsilon_k)$. Also, for each $k$ in $\N$, we define the natural number $j(k)$ as the unique solution to the equation $f(q_k) = q_{j(k)}$.

\subsubsection{Construction of $F$}
We now construct $F : \D \to \D$. For $x$ in $\D \setminus L$, let
\begin{equation} \label{DefOnDMinusL}
F(x) = \pi^{-1} \circ f \circ \pi(x).
\end{equation}
Since $L = \pi^{-1}(Q)$ and $Q$ is completely invariant for $f$, we have that if $x$ is in $\D \setminus L$ then $F(x)$ is in $\D \setminus L$. Note that $F$ is continuous on $\D \setminus L$ as it is a composition of continuous functions (recall that $\pi^{-1}|_{\D \setminus Q}$ is continuous by Lemma \ref{BlowUpLemma} (3)). We now show that the function $F$ can be extended to a continuous map on $\D \setminus \interior(L)$ such that $F(\partial B(p_k, \epsilon_k)) = \partial B(p_{j(k)}, \epsilon_{j(k)})$.

Suppose $d=1$ (the case $d = 2$ is treated below). Then $\partial B(p_k, \epsilon_k)$ is just the two endpoints of an interval. Because $Df_{q_k}$ is invertible, $f$ is either orientation preserving or orientation reversing at $q_k$. In either case, we extend $F$ continuously at $\partial B(p_k,\epsilon_k)$ so that $F$ maps $\partial B(p_k, \epsilon_k)$ bijectively to $\partial B(p_{j(k)}, \epsilon_{j(k)})$. Now we extend $F$ to the one-dimensional annulus $\{x : \frac{1}{2} \epsilon_k \leq |x-p_k| \leq \epsilon_k\}$ as follows. Let $T^{+} : [-1,-\frac{1}{2}] \cup [\frac{1}{2}, 1] \rightarrow [-1,-\frac{1}{2}] \cup [\frac{1}{2}, 1]$ be given by $T^{+}(x) = x + \frac{1}{10} \sin( 2 \pi |x|)$. Also, let $T^{-} = - x + \frac{1}{10} \sin( 2 \pi |x|)$. If $Df_{q_k} > 0$, let $\sigma = +$, and if $Df_{q_k} < 0$, let $\sigma=-$. Then for $x$ such that $\frac{1}{2} \epsilon_k \leq |x-p_k| \leq \epsilon_k$, let
\begin{equation} \label{DefOnAnnuluskdOne}
F(x) = \bigl( A_{\epsilon_{j(k)}, p_{j(k)}} \circ T^{\sigma} \circ A_{\epsilon_k, p_k}^{-1} \bigr)(x),
\end{equation}
where $A_{s,x}$ is defined in Notation \ref{RnNotation}. We remark that the additional terms involving sine in the functions $T^+$ and $T^-$ are introduced for technical convenience in proving Claim \ref{claimFive}.

Now suppose $d = 2$. We have that $\det Df|_Q > 0$, which implies that for each $k$, we may extend $F$ continuously on $\partial B(p_k,\epsilon_k)$  in the following way. There is an orientation preserving homeomorphism $T_k$ of the unit circle such that for $x$ in $\partial B(p_k, \epsilon_k)$, we let $F(x) = (A_{\epsilon_{j(k)},p_{j(k)}} \circ T_k \circ A_{\epsilon_k,p_k}^{-1})(x)$. Recall that any orientation preserving homeomorphism of the unit circle to itself is homotopic to the identity. Let $H_k: [\frac{1}{2},1] \times S^1$ be a homotopy such that $H(\frac{1}{2}, \cdot) = \Id$, $H_k(1, \cdot) = T_k$, and $H(t,\cdot)$ is a homeomorphism for each $t$ in $[\frac{1}{2},1]$. Now we extend $F$ to the annulus $\{x : \frac{1}{2} \epsilon_k \leq |x-p_k| \leq \epsilon_k\}$ as follows. We consider the annulus centered at $0$ with inner radius $\frac{1}{2}$ and outer radius $1$ in polar coordinates: $\{ (r,\theta) : r \in [\frac{1}{2},1], \theta \in S^1\} \subset \R^2$. For $(r,\theta)$ in this annulus, define $U_k(r, \theta) = \bigl(r + \frac{1}{10}\sin(2 \pi r), \, H_k(r,\theta)\bigr)$. Now for $x$ in $\D$ with $\frac{1}{2} \epsilon_k \leq |x-p_k| \leq \epsilon_k$, let
\begin{equation} \label{DefOnAnnuluskdTwo}
F(x) = \bigl( A_{\epsilon_{j(k)}, p_{j(k)}} \circ U_k \circ A_{\epsilon_k, p_k}^{-1} \bigr)(x).
\end{equation}

Up to this point in the construction, we have defined $F$ on $\D \setminus \cup_k B(p_k, \frac{1}{2}\epsilon_k)$. Now let $m$ be in $\N$ and suppose $\theta_m = \{q_{k_0}, \dots, q_{k_{|\theta_m|-1}}\}$. Let $g_{k_{|\theta_m|-1}} : \D \rightarrow \D$ be $\chi_m^{\xi_m}$, and let $g_{k_i}$ be the identity map on $\D$ for all $i \in \{0, \dots, |\theta_m|-2\}$. Making these choices for all $m$, we define $g_k$ for all $k$ such that $q_k$ is in $\cup_m \theta_m$. For all $k$ such that $q_k$ is not in $\cup_m \theta_m$, let $g_k$ be the identity map on $\D$. Now for each $k$ in $\N$ and $x$ in $B(p_k, \frac{1}{2}\epsilon_k)$, let
\begin{equation} \label{DefOnBkHalf}
F(x) = A_{\frac{1}{2}\epsilon_{j(k)},p_{j(k)}} \circ g_k \circ A_{\frac{1}{2}\epsilon_k,p_k}^{-1}(x).
\end{equation}
This concludes the construction of $F$.

\subsubsection{Properties of $F$} In this section we prove that $F$ has properties (1)-(9) in Proposition \ref{BlowAndSewProp}. For the sake of notation, we make some definitions. Let $K_0 = \D \setminus \interior(L)$, as in the statement of the proposition. For each $m$ in $\N$, let $K_m = \cup_{q_k \in \theta_m} \overline{B(p_k,\frac{1}{2}\epsilon_k)}$. The following claim follows directly from the construction of $F$.

\begin{claim}[Part of property (1)] \label{claimOne}
$F$ is a continuous surjection, and if $d = 2$, then $F$ is a homeomorphism. Also, $F|_{\partial \D} = \Id$.
\end{claim}

\begin{claim}[Property (2)] \label{claimTwo}
$\pi$ is a factor map from $(\D,F)$ to $(\D,f)$.
\end{claim}
\begin{proof}
By Lemma \ref{BlowUpLemma}, the map $\pi$ is continuous and surjective. For $x$ in $\D \setminus L$, we have that $\pi(F(x)) = f( \pi(x))$ by definition (Equation (\ref{DefOnDMinusL})). For $x$ in $\overline{B(p_k,\epsilon_k)}$, we have that $F(x)$ is in $\overline{B(p_{j(k)},\epsilon_{j(k)})}$ by definition, and then $\pi(F(x)) = q_{j(k)} = f(q_k) = f(\pi(x))$, using property (3) in Lemma \ref{BlowUpLemma}.
\end{proof}

\begin{claim}[Property (3)] \label{claimThree}
We have that $\pi(K_m)=\theta_m$ for each $m$ in $\N$, $L_k$ is $C^{\infty}$ diffeomorphic to $\D$, $\pi|_{\D \setminus \interior(L)}$ is injective, and $K_m \subset \interior(\cup_{q_k \in \theta_m} L_k)$.
\end{claim}
\begin{proof}
Let $\theta_m = \{q_{k_1}, \dots, q_{k_{\theta_m}}\}$. Then by property (3) in Lemma \ref{BlowUpLemma}, we have $\pi(K_m) = \pi(\cup_{q_k \in \theta_m} \overline{B(p_k,\frac{1}{2}\epsilon_k)}) = \theta_m$. The second assertion follows immediately from the fact that $L_k = \overline{B(p_k,\epsilon_k)}$, and the third assertion holds since $K_m = \cup_{q_k \in \theta_m} \overline{B(p_k,\frac{1}{2}\epsilon_k)} \subset \cup_{q_k \in \theta_m} L_k$.
\end{proof}

The following claim follows directly from the construction.
\begin{claim}[Property (4)] \label{claimFour}
$K_i$ is $F$-invariant for each $i$ in $\Z_{\geq 0}$, $K_0 = \D \setminus \interior(L)$, and $\cup_k L_k$ is $F$-invariant.
\end{claim}

\begin{claim}[Property (5)] \label{claimFive}
$\NW(F) \subseteq \bigcup_{i=0}^{\infty} K_i$.
\end{claim}
\begin{proof}
If $x$ is in $B(p_k,\epsilon_k)$ for some $k$ such that $q_k$ is not periodic, then $x$ is wandering because $q_k$ is pre-periodic. Now consider the periodic orbit $\theta_m$. Recall that any point in $(\frac{1}{2},1)$ is wandering for the map $T(t) = t + \frac{1}{10} \sin(2 \pi t - \pi)$. According to Equations (\ref{DefOnAnnuluskdOne}) and (\ref{DefOnAnnuluskdTwo}), the radial component of $F$ restricted to $\cup_{q_k \in \theta_m} B(p_k, \epsilon_k) \setminus \overline{B(p_k,\frac{1}{2}\epsilon_k)}$ is conjugate to a tower over $T$. It follows that any $x$ in $\cup_{q_k \in \theta_m} B(p_k, \epsilon_k) \setminus \overline{B(p_k,\frac{1}{2}\epsilon_k)}$ is wandering, which means that $\NW(F) \subset (K_0) \cup \Bigl( \bigcup_m K_m \Bigr)$.
\end{proof}

\begin{claim}[Property (6)] \label{claimSix}
$F|_{K_0}$ is a principal extension of $f$ with factor map $\pi|_{K_0}$, and for $\nu$ in $M(K_0 \setminus \cup_k \partial L_k, F)$, it holds that $\pi$ is a measure theoretic isomorphism between $(F,\nu)$ and $(f,\pi(\nu))$.
\end{claim}
\begin{proof} Let $\nu$ be in $M(K_0 \setminus \cup_k \partial L_k, F)$. By conclusion (3) of Lemma \ref{BlowUpLemma}, the factor map $\pi$ is injective on $K_0 \setminus \cup_k L_k$ and therefore defines a measure theoretic isomorphism between $(F,\nu)$ and $(f,\pi(\nu))$. An ergodic measure $\nu$ for $F|_{K_0}$ that is not in $M(K_0 \setminus \cup_k \partial L_k, F)$ has $\nu(\cup_k \partial L_k) = 1$, and therefore $h^F(\nu) = 0$. It follows that for every $\nu$ in $M(K_0,F)$, we have $h^F(\nu) = h^f(\pi(\nu))$.

\end{proof}

Let $\theta_m = \{q_{k_0}, \dots , q_{k_{|\theta_m|-1}}\}$ be a periodic orbit for $f$ labeled such that $f(q_{k_i}) = q_{k_{i+1}}$, where $i+1$ is taken modulo $|\theta_m|$. Let $\phi_m : \D \times \{0, \dots, |\theta_m|-1\} \to \cup_{i=0}^{|\theta_m|-1} \overline{ B(p_{k_i},\frac{1}{2}\epsilon_{k_i}) }$ be the map given by $\phi_m(x,i) = A_{\frac{1}{2}\epsilon_{k_i},p_{k_i}}(x)$.

\begin{claim}[Property (7)] \label{claimSeven}
 $F|_{K_m}$ is topologically conjugate by the map $\phi_m$ to a $(|\theta_m|,\xi_m)$ tower over $\chi_m$, for each $m$ in $\N$.
\end{claim}
\begin{proof}
By Equation (\ref{DefOnBkHalf}), for any $m$ and any $k$ such that $q_k$ is in $\theta_m$, $F|_{\overline{B(p_k,\frac{1}{2}\epsilon_k)}} = A_{\frac{1}{2}\epsilon_{j(k)},p_{j(k)}} \circ g_k \circ A_{\frac{1}{2}\epsilon_k,p_k}^{-1}$. Then by the choice of $g_k$, we have that $F$ is topologically conjugate to a $(|\theta_m|,\xi_m)$ tower over $\chi_m$, with the conjugacy given by the map $\phi_m$.
\end{proof}

\begin{claim}[Property (8)] \label{claimEight}
 Let $C_0 = \cap_{n=1}^{\infty} \overline{ \cup_{m \geq n} M(K_m,F) }$ and also let $C(f) = \cap_{n=1}^{\infty} \overline{ \cup_{m \geq n} \{\mu_{\theta_m}\} }$. Then $C_0 = \pi^{-1}(C(f) ) \subset M(K_0 \setminus \cup_k \partial L_k,F)$, and $\pi$ maps $C_0$ homeomorphically onto $C(f)$.
\end{claim}
\begin{proof}
Let $\{\mu_{m_{\ell}}\}_{\ell}$ be a sequence of measures in $M(\D,F)$ tending to $\mu$ such that $\mu_{m_{\ell}} \in M(K_{m_{\ell}},F)$ for each $\ell$. Then the sequence $\{\pi(\mu_{m_{\ell}}) = \mu_{\theta_{m_{\ell}}}\}_{\ell}$ converges to $\pi(\mu)$ by the continuity of $\pi$, which shows that $C_0 \subset \pi^{-1}(C(f))$.

Now let $\mu$ be in $C(f)$, and let $\nu$ be in $\pi^{-1}(\mu)$. By property (1) in the definition of the statement that $f$ is ready for operation on $\cup_m \theta_m$ (Definition \ref{readyForOpDef}), $\mu(\cup_m \theta_m)=0$, and thus $\nu(L) = 0$. Therefore $\nu \in M(K_0 \setminus \cup_k \partial L_k,F)$, and we have shown that $\pi^{-1}(C(f)) \subset M(K_0 \setminus \cup_k \partial L_k,F)$. Since $\pi|_{\D \setminus \cup_k L_k}$ is a homeomorphism onto its image $\D \setminus Q$, we also have that for any $\mu$ in $C(f)$, the set $\pi^{-1}(\mu)$ consists of exactly one measure.

Now let $\mu_{\theta_{m_k}}$ converge to $\mu$ in $M(\D,f)$. By the previous statement, there exists a measure $\nu$ such that $\{ \nu \} = \pi^{-1}(\mu)$. Now choose any sequence of measures $\{\nu_{m_k}\}_k$ such that $\nu_{m_k}$ is in $\pi^{-1}(\mu_{\theta_{m_k}})$ for each $k$. By the sequential compactness of $M(\D,F)$, any subsequence $\{\tau_n\}_n$ of $\{\nu_{m_k}\}_k$ has a subsequence $\{\tau_{n_{\ell}}\}_{\ell}$ that converges to some measure $\tau$. By continuity of $\pi$, we have $\pi(\tau) = \mu$. Since $\pi^{-1}(\mu) = \{\nu\}$, we see that $\tau = \nu$. Since this holds for any subsequence of $\{\nu_{m_k}\}_k$, it follows that $\{\nu_{m_k}\}_k$ converges to $\nu$. This argument shows that $C_0 \supset \pi^{-1}(C(f))$, and therefore $C_0 = \pi^{-1}(C(f))$ (since we showed the reverse inclusion at the beginning of this proof). Since $\pi$ is surjective, we also have that $\pi(C_0) = C(f)$.

Now we have that $\pi|_{C_0}$ is a continuous bijective map from a compact space into a Hausdorff space. It follows that $\pi$ maps $C_0$ homeomorphically onto its image $C(f)$, which completes the proof.
\end{proof}

\begin{claim}[Property (1)] \label{claimTen}
$F$ is in $\Cl_d$.
\end{claim}
\begin{proof}
Claims \ref{claimFive}-\ref{claimSeven} and the variational principle imply that
\begin{equation*}
\htop(F) = \max \Bigl( \htop(f), \; \sup_{m} \frac{\xi_m}{|\theta_m|} \htop( \chi_m) \Bigr).
\end{equation*}
The right-hand side of this equation is finite by hypothesis. Combining this fact with Claim \ref{claimOne}, we obtain that $F$ is in $\Cl_d$.
\end{proof}

\begin{claim}[Property (9)] \label{claimNine}
 $F$ is ready for operation on $\cup_{m, \ell} \; \cup_{q_k \in \theta_m} A_{\frac{1}{2} \epsilon_k, p_k} (\theta_{\ell}^m )$.
\end{claim}
\begin{proof}
First note that $F$ is in $\Cl_d$ by Claim \ref{claimTen}. Also, we have that $S = \cup_{m, \ell} \; \cup_{q_k \in \theta_m} A_{\frac{1}{2} \epsilon_k, p_k} (\theta_{\ell}^m )$ is a countable collection of periodic points for $F$ by Claim \ref{claimSeven}. Let $\Theta_i$ be an enumeration of the periodic points orbits in $S$, and let $C(F) = \cap_{n=1}^{\infty} \overline{\cup_{k \geq n} \{\mu_{\Theta_i}\} }$. Now we check that $F$ satisfies the properties (1)-(3) in Definition \ref{readyForOpDef}.

Let $\nu$ be in $C(F)$. Let $C_0 = \cap_{n=1}^{\infty} \overline{ \cup_{m\geq n} M(K_m,F) }$. Note that $C(F) \subset C_0 \cup \Bigl(\cup_{m \geq 1} M(K_m,F)\Bigr)$. By Claim \ref{claimEight}, we have that $C_0 \subset M(K_0 \setminus \cup_k \partial L_k,F)$. Thus if $\nu$ is in $C_0$, then $\nu(L) = 0$. Since $\cup_i \Theta_i \subset \cup_{m \geq 1} K_m \subset L$, it follows that $\nu(\cup_i \Theta_i)=0$, which proves property (1) in Definition \ref{readyForOpDef} in the case that $\nu$ is in $C_0$. Now suppose $\nu$ is in $M(K_m,F)$ for some $m \geq 1$. By Claim \ref{claimSeven}, we have that $F|_{K_m}$ is topologically conjugate to a tower over $\chi_m$ via the map $\phi_m$. Any sequence $\{\Theta_{i_k}\}$ such that $\{\mu_{\Theta_{i_k}}\}$ converges to $\nu$ must eventually lie in $K_m$, and therefore $\nu( \cup_i \Theta_i) = 0$ because $\chi_m$ is ready for operation on $\cup_{\ell} \theta^m_{\ell}$.

To check that $F$ satisfies property (2) in Definition \ref{readyForOpDef}, we note that $Q = \cup_{i,k} F^{-k}(\Theta_i)$ is countable and contained in $\interior(\D)$ because $f$ and $\chi_m$ satisfy these properties with their respective sequences of periodic points, $\{\theta_m\}_m$ and $\{\theta_{\ell}^m\}_{\ell}$.

To check Property (3) in Definition \ref{readyForOpDef}, we need to check that $DF|_x$ is continuous and invertible at each point $x$ of $Q$ and that $\det DF|_x >0$ if $d = 2$. For each point $x$ in $Q$ there is an open set $B(p_k,\epsilon_k)$ containing $x$ on which $F$ is either affine or conjugate by affine maps to a tower over $\chi_m$. Property (3) in Definition \ref{readyForOpDef} is satisfied at $x$ if $F$ is affine on $B(p_k,\epsilon_k)$. If $F$ is conjugate to a tower over $\chi_m$ on $B(p_k,\epsilon_k)$, then $F$ satisfies property (3) of Definition \ref{readyForOpDef} because $\chi_m$ satisfies this property, which extends to simple towers.
\end{proof}

\subsubsection{Conclusion of the proof of Proposition \ref{BlowAndSewProp}}
By Claims \ref{claimOne}-\ref{claimTen}, properties (1)-(9) are satisfied for $F$, $\pi$, $\{K_i\}_{i=0}^{\infty}$, and $\{\phi_m\}_{m \in \N}$. This completes the proof of Proposition \ref{BlowAndSewProp}.
\end{proof}

\subsection{Additional properties of the blown-up map}

\begin{defn}
Let $f$, $\{\theta_m\}_{m \in \N}$, $\{\chi_m\}_{m \in \N}$, $\{\theta_{\ell}^m\}_{m,\ell \in \N}$, and $\{\xi_m\}_{m \in \N}$ be as in the hypotheses of Proposition \ref{BlowAndSewProp}. We define $\Bl( f, \{\theta_m\}_m, \{\chi_m\}_m, \{\theta_{\ell}^m\}_{m,\ell}, \{\xi_m\}_m )$ to be the set of functions $F$ in $\Cl_d$ such that there exists $\pi$, $\{K_i\}_i$, and $\{\phi_m\}_m$ as in the statement of Proposition \ref{BlowAndSewProp}. In these terms, Proposition \ref{BlowAndSewProp} asserts that $\Bl( f, \{\theta_m\}_m, \{\chi_m\}_m, \{\theta_{\ell}^m\}_{m,\ell}, \{\xi_m\}_m )$ is non-empty.
\end{defn}

%%%%%%%%%%%%%%%%%%%%%%%%%%%%%%%%%%%% Entropy structure of F %%%%%%%%%%%%%%%%%%%%%%%%%%%%%%%%%%%%%%%%%%%%%%%%%%%%%%%%%%%%%%%%

\begin{lem} \label{ESzeroLemma}
Let $F : D \rightarrow D$ be a continuous surjection of a compact metric space. Suppose that $\NW(F) \subseteq \sqcup_{i=0}^{\infty} K_i$, where each $K_i$ is compact, $F(K_i)=K_i$, and $K_i = \cup_{j=1}^{J_i} K_i^j$, where the sets $\{K_i^j\}_{j=1}^{J_i}$ are compact and pairwise disjoint. Also, suppose that $\lim_i \max_{1 \leq j \leq J_i} \diam(K_i^j) = 0$. Then there exists an entropy structure $(f_k)$ for $F$ with the following property: for each $k$, there exists $I$ such that if $i > I$ then $f_k|_{M(K_i,F)} \equiv 0$.
\end{lem}
\begin{proof}
Let $(f_k)$ be the Katok entropy structure (see Definition \ref{KatokESDef}) corresponding to a sequence $\{\epsilon_k\}_k$ of positive numbers that tends to $0$. Let $k$ be given. Since $\lim_i \max_{1 \leq j \leq J_i} \diam(K_i^j) = 0$, there exists $I$ such that $i > I$ implies that $\diam(K_i^j) < \epsilon_k$ for all $1 \leq j \leq J_i$. Then for $i > I$ and ergodic $\mu$ such that $\supp(\mu) \subset K_i$, we have that $h^F(\mu, \epsilon_k,\sigma) = 0$ because $K_i$ is invariant and $\diam(K_i^j) < \epsilon_k$ for $1 \leq j \leq J_i$. Since this holds for ergodic measures $\mu$ with $\supp(\mu) \subset K_i$, it also holds for any invariant measure $\mu$ with $\supp(\mu) \subset K_i$ because $f_k$ is harmonic, which completes the proof.
\end{proof}

\begin{lem} \label{ESsynchronizationLemma}
 Let $F : D \rightarrow D$ be a continuous surjection of a compact metric space. Suppose that $\NW(F) \subseteq \sqcup_{i=0}^{\infty} K_i$, where each $K_i$ is compact, $F(K_i)=K_i$, and $K_i = \cup_{j=1}^{J_i} K_i^j$, where the sets $\{K_i^j\}_{j=1}^{J_i}$ are compact and pairwise disjoint. Also, suppose that $\lim_i \max_{1 \leq j \leq J_i} \diam(K_i^j) = 0$. For each $i$ in $\Z_{\geq 0}$ fix a harmonic entropy structure $\H^i=(h_{\ell}^i)$ for $F|_{K_i}$.

 Then there exists a harmonic entropy structure $\H(F) = (h_k^F)$ such that $h_k^F(\mu) = h_k^0(\mu)$ for $\mu$ with $\supp(\mu) \subset K_0$, and for every $i$ in $\N$, there is a non-decreasing function $\ell_i : \Z_{\geq 0} \rightarrow \Z_{\geq 0}$ with the following properties:
\begin{enumerate}
 \item if $\mu$ is in $M(D,F)$ and $\supp(\mu) \subset K_i$, then $h_k^F(\mu) = h_{\ell_i(k)}^{i}(\mu)$ for every $k$ in $\Z_{\geq 0}$.
 \item for any $k$ in $\N$, there exists $I$ in $\N$ such that $\ell_i(k) = 0$ for all $i \geq I$.
\end{enumerate}
\end{lem}
\begin{proof}
Let $\F = (f_k)$ be a harmonic entropy structure for $F$ with the property that for every $k$ there exists $I$ such that if $i > I$ then $f_k|_{M(K_i,F)} \equiv 0$ (such entropy structures exist by Lemma \ref{ESzeroLemma}). Let $\delta_k >0$ be a sequence tending to $0$. Let $i$ be in $\N$. Since $(f_k|_{M(K_i,F)})$ and $(h_{\ell}^i)$ are both an entropy structures for $F|_{K_i}$, we have that $(f_k|_{M(K_i,F)})$ and $(h_{\ell}^{i})$ are uniformly equivalent. Using the definition of uniform equivalence (in particular the fact that $(f_k|_{M(K_i,F)})$ is uniformly dominated by $(h_{\ell}^{i})$), we define $\ell_i(k) = \min \{ \ell \geq 0 : h_{\ell}^{i} \geq f_k|_{K_i} - \delta_k\}$ for each $k$ in $\Z_{\geq 0}$. By construction, $\ell_i$ is non-decreasing. For ergodic measures $\mu$ in $M(K_i,F)$, let $h_k^F(\mu) = h_{\ell_i(k)}^{i}(\mu)$. For ergodic $\mu$ in $M(K_0,F)$, let $h_k^F(\mu) = h_k^0(\mu)$.

Since every ergodic measure for $F$ is in $\cup_i M(K_i,F)$, we have defined $h_k^F$ for all ergodic measures. Define $h_k^F$ on all non-ergodic measures by harmonic extension, and let $\H(F)=(h_k^F)$. Note that since $h_{\ell_i(k)}^{i} $ is harmonic, for $\mu$ in $M(K_i,F)$, we have that $h_k^F(\mu) = h_{\ell_i(k)}^{i}(\mu)$ (which shows that if $(h_k^F)$ is an entropy structure, then it satisfies property (1) by definition).  By construction, $\H(F)$ is harmonic. It remains to check that $\H(F)$ is an entropy structure for $F$.

We show that $\H(F)$ is uniformly equivalent to $\F$, which implies that $\H(F)$ is an entropy structure for $F$. Since $\F$ and $\H(F)$ are harmonic, we may restrict attention to ergodic measures. Fix $k$ and $\epsilon > 0$, and choose $k' \geq k$ large enough that $\delta_{k'} < \epsilon$. Then for every ergodic $\mu$, we have that $\mu$ is in some $M(K_i,F)$, and $h_{k'}^F(\mu) = h_{\ell_i(k')}^{i}(\mu) \geq f_{k'}(\mu) - \delta_{k'} \geq f_k(\mu) - \epsilon$. Hence $\H(F) \geq \F$. Again, fix $k$ and $\epsilon >0$. Choose $I$ such that $f_k|_{M(K_i,F)} \equiv 0$ for all $i > I$ (such an $I$ exists by the choice of the sequence $(f_k)$). Then it follows from the definition of $\ell_i(k)$ that $\ell_i(k)=0$ for all $i > I$ (showing property (2)). Using that $(f_k|_{M(K_i,F)})$ and $(h_{\ell}^{i})$ are uniformly equivalent for each $i \leq I$ (in particular, $(f_k|_{M(K_i,F)})$ uniformly dominates $(h_{\ell}^{i})$), there exists $k_i$ such that $f_{k_i}|_{M(K_i,F)} \geq h_{\ell_i(k)}^{i} - \epsilon$. Let $k' = \max(k_0, \dots, k_I)$. Any ergodic measure $\mu$ is in $M(K_i,F)$ for some $i$. Let $\mu$ be ergodic in and contained in $M(K_i,F)$. If $i \leq I$, then $f_{k'}(\mu) \geq f_{k_i}(\mu) \geq h_{\ell_i(k)}^{i} ( \mu) - \epsilon = h_{k}^F(\mu) - \epsilon$. If $i > I$, then $f_{k'}(\mu) = 0 \geq -\epsilon = h_{\ell_i(k)}^{i}(\mu) - \epsilon = h_k^F(\mu) - \epsilon$. Since these same bounds hold for all ergodic $\mu$, we have that $f_{k'} \geq h_k^F - \epsilon$, and we have shown that $\F$ uniformly dominates $\H(F)$. Then $\F$ and $\H(F)$ are uniformly equivalent, and we conclude that $\H(F)$ is an entropy structure for $F$. This concludes the proof of the lemma.
\end{proof}

\begin{prop} \label{ESofFProp}
Let $f$, $\{\theta_m\}_{m \in \N}$, $\{\chi_m\}_{m \in \N}$, $\{\theta_{\ell}^m\}_{m,\ell \in \N}$,  $\{\xi_m\}_{m \in \N}$, $F$, $\pi$, $\{K_i\}_{i}$, and $\{\phi_m\}_m$ all be as in Proposition \ref{BlowAndSewProp}. For each $m$ in $\N$, let $S_m = \phi_m^{-1} \circ F|_{K_m} \circ \phi_m$ and let $\psi_m$ be the map associated to the tower $S_m$ over $\chi_m$ (Definition \ref{TowerMapDef}). For each $m$ in $\N$, let $\H(\chi_m) = (h_k^{\chi_m})$ be a harmonic entropy structure for $\chi_m$, and let $\H(f) = (h_k^f)$ be a harmonic entropy structure for $f$. Then there exists a harmonic entropy structure $\H(F) = (h_k^{F})$ for $F$ such that
\begin{enumerate}
 \item for $\mu$ with $\supp(\mu) \subset K_0$, $h_k^{F}(\mu) = h_k^f(\pi(\mu))$;
 \item for every $m$ and $k$, there exists $k'$ such that for each $\mu$ with $\supp(\mu) \subset K_m$, $h_k^{F}(\mu) = \frac{\xi_m}{|\theta_m|} h_{k'}^{\chi_m}( \psi_m ( (\phi_m^{-1})(\mu)))$;
 \item for every $k$ there exists $m_0$ such that if $m \geq m_0$ and $\supp(\mu) \subset K_m$, then $h_k^F(\mu) = 0$.
\end{enumerate}
\end{prop}
\begin{proof}
By Fact \ref{PrincExtFact}, $(h_k^f \circ \pi)$ is an entropy structure for $F|_{K_0}$. By Lemma \ref{TowerLemmaOne}, $(\frac{\xi_m}{|\theta_m|}h_k^{\chi_m} \circ \psi_m)$ is an entropy structure for $S_m$. Since $\phi_m$ is a topological conjugacy between $S_m$ and $F|_{K_m}$, we have that $(\frac{\xi_m}{|\theta_m|}h_k^{\chi_m} \circ \psi_m \circ (\phi_m^{-1}))$ is an entropy structure for $F|_{K_m}$. Then Lemma \ref{ESsynchronizationLemma} gives that these entropy structures can be combined to form an entropy structure for $F$ satisfying properties (1)-(3).
\end{proof}

The following corollary is a consequence of Proposition \ref{ESofFProp}, but one may also check it directly as in the proof of Claim \ref{claimTen}.
\begin{cor} \label{HtopCor}
Let $f$, $\{\theta_m\}_{m \in \N}$, $\{\chi_m\}_{m \in \N}$, $\{\theta_{\ell}^m\}_{m,\ell \in \N}$, and $\{\xi_m\}_{m \in \N}$ be as in the hypotheses of Proposition \ref{BlowAndSewProp}. Let $F$ be in $\Bl( f, \{\theta_m\}_m, \{\chi_m\}_m, \{\theta_{\ell}^m\}_{m,\ell}, \{\xi_m\}_m )$. Then
\begin{equation*}
 \htop(F) = \max \Bigl( \htop(f), \; \sup_{m} \frac{\xi_m}{|\theta_m|} \htop( \chi_m) \Bigr)
\end{equation*}
\end{cor}

The following lemma is used to compute the transfinite sequence associated to some of the systems in Section \ref{TwoLemmas}. In this lemma we combine our lemma for principal extensions (Lemma \ref{PrincExtLemma}) and our lemma for towers (Lemma \ref{TowerLemmaThree}) with our analysis of the ``blow-and-sew" construction (Proposition \ref{BlowAndSewProp}) to give a precise description of the measures and transfinite sequences of some maps constructed by the ``blow-and-sew" operation.
\begin{lem} \label{totallyErgLemma}
Let $f$, $\{\theta_m\}_{m \in \N}$, $\{\chi_m\}_{m \in \N}$, $\{\theta_{\ell}^m\}_{m,\ell \in \N}$,  $\{\xi_m\}_{m \in \N}$, $F$, $\pi$, $\{K_i\}_{i}$, and $\{\phi_m\}_m$ all be as in Proposition \ref{BlowAndSewProp}. Let $\{\Theta_k\}_k$ be an enumeration of the $F$-periodic orbits in $\cup_{m,\ell} \, \phi_m(\theta_{\ell}^m \times \{0, \dots, |\theta_m|-1\})$. Let $M = \cup_i M(K_i,F)$, $C(f) = \cap_{n=1}^{\infty} \overline{ \cup_{m\geq n} \{\mu_{\theta_m}\} }$, $C(\chi_m) = \cap_{n=1}^{\infty} \overline{ \cup_{\ell \geq n} \{\mu_{\theta_{\ell}^m}\} }$, and $C(F) = \cap_{n=1}^{\infty} \overline{ \cup_{k \geq n} \{\mu_{\Theta_k}\} }$. Suppose that
\renewcommand{\labelenumi}{(\roman{enumi})}
\begin{enumerate}
 \item each $\mu$ in  $C(f)$ is totally ergodic for $f$;
 \item for each $\mu$ in $M(\D,f)$ and each ordinal $\gamma$,
\begin{equation*}
u_{\gamma}^{\H(f)}(\mu) = \int_{C(f)} u_{\gamma}^{\H(f)|_{C(f)}} \; d\P_{\mu};
\end{equation*}
 \item each $\mu$ in $C(\chi_m)$ is totally ergodic for $\chi_m$;
 \item for each $\mu$ in $M(\D,\chi_m)$ and each ordinal $\gamma$,
\begin{equation*}
u_{\gamma}^{\H(\chi_m)}(\mu) = \int_{C(\chi_m)} u_{\gamma}^{\H(\chi_m)|_{C(\chi_m)}} \; d\P_{\mu};
\end{equation*}
 \item either $\htop(F|_{K_m})$ tends to $0$ as $m$ tends to infinity, or for each $m \geq 1$, $\al_0(F|_{K_m}) = 0$ and $\htop(F|_{K_m}) = h^{F|_{K_m}}(\mu)$ for $\mu$ in $C(F) \cap M(K_m,F)$.
\end{enumerate}
Then
\renewcommand{\labelenumi}{(\arabic{enumi})}
\begin{enumerate}
\item each measure $\nu$ in $C(F)$ is totally ergodic for $F$;
\item $C(F) = \pi^{-1}(C(f)) \cup \Bigl( \bigcup_m \phi_m(\psi_m^{-1}(C(\chi_m))) \Bigr) $;
\item $\pi$ maps $C(F) \cap M(K_0,F)$ homeomorphically onto $C(f)$;
\item $\psi_m \circ (\phi_m)^{-1}$ maps $C(F) \cap M(K_m,F)$ homeomorphically onto $C(\chi_m)$.
\item for all $x$ in $M$ and all ordinals $\gamma$,
\begin{equation*}
u_{\gamma}^{\H(F)|_M}(x) \leq \int_{C(F)} u_{\gamma}^{\H(F)|_{C(F)}} \; d\P_x.
\end{equation*}
\end{enumerate}
\end{lem}
\begin{proof}
Let $\nu$ be in $C(F)$. Since $\cap_{n=1}^{\infty} \overline{ \cup_{m \geq n} M(K_m,F)} \subset M(K_0,F)$ (conclusion (8) in Proposition \ref{BlowAndSewProp}), we have that $\nu$ is in $M(K_i,F)$ for some $i$.

Suppose $\nu$ is in $M(K_0,F)$. Then $\nu$ is in $M(K_0 \setminus \cup_k \partial L_k,F)$ and $\pi(\nu)$ is in $C(f)$ by conclusion (8) in Proposition \ref{BlowAndSewProp}. By conclusion (6) in Proposition \ref{BlowAndSewProp}, we have that $\pi$ gives a measure preserving isomorphism between $\nu$ and $\pi(\nu)$. The fact that $\nu$ is totally ergodic now follows from the hypothesis that $\pi(\nu)$ is totally ergodic (since it is in $C(f)$).

Now suppose that $\nu$ is in $M(K_m,F)$ for some $m$ in $\N$. By conclusion (7) in Proposition \ref{BlowAndSewProp}, the map $\phi_m$ is a topological conjugacy between $F|_{K_m}$ and a $(|\theta_m|,\xi_m)$ tower over $\chi_m$. By Lemma \ref{TowerLemmaThree}, $(\phi_m^{-1})(\nu)$ is totally ergodic, and therefore $\nu$ is totally ergodic, proving (1).

Property (3) is contained in conclusion (8) of Proposition \ref{BlowAndSewProp}. Using that $\phi_m$ is a topological conjugacy between $F|_{K_m}$ and a tower over $\chi_m$, we obtain property (4) from Lemma \ref{TowerLemmaThree} (2). Then property (2) follows from properties (3) and (4) the fact that $C(F) = \cup_{i \geq 0} M(K_i,F) \cap C(F)$.

Now we prove (5). First note that for $m \geq 1$, $M(K_m,F)$ is open in $M$, since $\cap_{n=1}^{\infty} \overline{ \cup_{m \geq n} M(K_m,F)} \subset M(K_0,F)$ (conclusion (8) in Proposition \ref{BlowAndSewProp}). Then Fact \ref{TransSeqFacts} (1) implies that for all $x$ in $M(K_m,F)$, $u_{\gamma}^{\H(F)|_M}(x) = u_{\gamma}^{\H(F)|_{M(K_m,F)}}(x)$. Furthermore, Lemma \ref{TowerLemmaThree} (2) and monotonicity (Fact \ref{TransSeqFacts} (2)) give that for $x$ in $M(K_m,F)$,
\begin{align}
u_{\gamma}^{\H(F)|_{M(K_m,F)}}(x) & = \int_{C(F) \cap M(K_m,F)} u_{\gamma}^{\H(F)|_{C(F) \cap M(K_m,F)}} \; d\P_x \label{intFormulaEqn} \\ & \leq \int_{C(F) } u_{\gamma}^{\H(F)|_{C(F)}} \; d\P_x,
\end{align}
which gives the desired inequality for all $x$ in $\cup_{m \geq 1} M(K_m,F)$.

Next, note that $M(K_0,F) \setminus  C(F)$ is open in $M$ (by Proposition \ref{BlowAndSewProp} (8)). Then Fact \ref{TransSeqFacts} (1) gives that for all $x$ in $M(K_0,F)$, $u_{\gamma}^{\H(F)|_M}(x) = u_{\gamma}^{\H(F)|_{M(K_0,F)}}(x)$. By Lemma \ref{PrincExtLemma} and Fact \ref{TransSeqFacts} (2), we obtain that for all $x$ in $M(K_0,F) \setminus C(F)$,
\begin{align*}
u_{\gamma}^{\H(F)|_{M(K_0,F)}}(x) & = \int_{C(F) \cap M(K_0,F)} u_{\gamma}^{\H(F)|_{C(F) \cap M(K_0,F)}} \; d\P_x \\ &  \leq \int_{C(F) } u_{\gamma}^{\H(F)|_{C(F)}} \; d\P_x,
\end{align*}
which gives the desired inequality for all $x$ in $M(K_0,F) \setminus C(F)$.

Lastly, we show (5) for all $x$ in $C(F) \cap M(K_0,F)$ using transfinite induction. Note that $C(F) \cap M(K_0,F) \subset M_{\erg}(\D,F)$, and therefore $\P_x$ is just the point mass at $x$. Thus for $x$ in $C(F) \cap M(K_0,F)$ property (5) is equivalent to $u_{\gamma}^{\H(F)|_M}(x) \leq u_{\gamma}^{\H(F)|_{C(F)}}(x)$. Property (5) holds trivially for $\gamma = 0$. Now suppose for the sake of induction it holds for an ordinal $\gamma$, and we show it holds for $\gamma+1$. For the sake of notation, let $M_i = M(K_i,F) \setminus C(F)$. Let $x$ be in $C(F) \cap M(K_0,F)$. Then using the induction hypothesis and our computation of the transfinite sequence for $y$ in $M \setminus (C(F) \cap M(K_0,F))$,
\begin{align*}
\limsup_{\substack{y \to x \\ y \in M }} & (u_{\gamma}^{\H(F)|_M} + \tau_k)(y)  \\
&= \max \Bigl(\limsup_{\substack{y \to x \\ y \in C(F) }} (u_{\gamma}^{\H(F)|_M}+ \tau_k)(y), \; \limsup_{\substack{y \to x \\ y \in M_0 }} (u_{\gamma}^{\H(F)|_M}+ \tau_k)(y), \\ & \quad \quad \quad \quad \quad \quad \quad \quad \quad \quad \quad \quad \quad \quad \quad \quad \quad \quad \limsup_{\substack{y \to x \\ y \in \cup_{m\geq 1} M_m }} (u_{\gamma}^{\H(F)|_M}+ \tau_k)(y) \Bigr) \\
&\leq \max \Bigl(\limsup_{\substack{y \to x \\ y \in C(F) }} (u_{\gamma}^{\H(F)|_{C(F)}}+ \tau_k)(y), \; \limsup_{\substack{y \to x \\ y \in M_0 }} (u_{\gamma}^{\H(F)|_{M(K_0,F)}}+ \tau_k)(y), \\ & \quad \quad \quad \quad \quad \quad \quad \quad \quad \quad \quad \quad \quad \quad \quad \quad \quad \quad  \limsup_{\substack{y_{m_{\ell}} \to x \\ y_{m_{\ell}} \in M_{m_{\ell}} } } (u_{\gamma}^{\H(F)|_{M(K_{m_{\ell}},F)}}+ \tau_k)(y_{m_{\ell}}) \Bigr).
\end{align*}
Letting $k$ tend to infinity in the above expressions gives that
\begin{align} \label{PenUltEqn1}
u_{\gamma+1}^{\H(F)|_M}(x) \leq \max\Bigl(  u_{\gamma+1}^{\H(F)|_{C(F)}}(x), & \; u_{\gamma+1}^{\H(F)|_{M(K_0,F)}}(x), \\ & \limsup_{\substack{y_{m_{\ell}} \to x \\ y_{m_{\ell}} \in M_{m_{\ell}}}} u_{\gamma}^{\H(F)|_{M(K_{m_{\ell}},F)}}(y_{m_{\ell}})+\htop(F|_{K_{m_{\ell}}})\Bigr) \label{PenUltEqn2}.
\end{align}
We would like to show that the expression in the right-hand side of Equation (\ref{PenUltEqn1})-(\ref{PenUltEqn2}) is less than or equal to $u_{\gamma+1}^{\H(F)|_{C(F)}}(x)$, and we prove this bound by analyzing each expression in the maximum individually. The bound is trivial for the first expression. By Lemma \ref{PrincExtLemma} (applied to $F|_{K_0}$, which is a principal extension of $f$, with $C(f)$ in place of $C$), we have that for $x$ in $M(K_0,F)$,
\begin{equation} \label{UsePrincExtLemma}
u_{\gamma+1}^{\H(F)|_{M(K_0,F)}}(x) = \int_{C(F) \cap M(K_0,F)} u_{\gamma+1}^{\H(F)|_{C(F) \cap M(K_0,F)}} \; d\P_x.
\end{equation}
Since $C(F) \subset M_{\erg}(\D,F)$, the measure $\P_x$ is the point mass at $x$ for any $x$ in $C(F)$. Combining this fact with Equation \ref{UsePrincExtLemma} and then using Fact \ref{TransSeqFacts} (2) gives that
\begin{equation*}
u_{\gamma+1}^{\H(F)|_{M(K_0,F)}}(x) = u_{\gamma+1}^{\H(F)|_{C(F) \cap M(K_0,F)}}(x) \leq u_{\gamma+1}^{\H(F)|_{C(F)}}(x),
\end{equation*}
which gives the desired bound on the second expression in the maximum in Equation (\ref{PenUltEqn1})-(\ref{PenUltEqn2}).

We bound the third expression in the maximum in Equation (\ref{PenUltEqn1})-(\ref{PenUltEqn2}) as follows. By hypothesis (v), either $\htop(F|_{K_m})$ tends to $0$ as $m$ tends to infinity or for each $m$, $\al_0(F_{K_m}) = 0$ and $\htop(F|_{K_m}) = h^F(\mu)$ for $\mu$ in $C(K_m,F)$. First suppose that $\htop(F|_{K_m})$ tends to $0$. Let $\{y_{m_{\ell}}\}_{\ell}$ be any sequence tending to $x$ such that $y_{m_{\ell}} \in M(K_{m_{\ell}},F)$ for each $\ell$. Equation (\ref{intFormulaEqn}) implies that $||u_{\gamma}^{\H(F)|_{M(K_{m_{\ell}},F)}}|| = ||u_{\gamma}^{\H(F)|_{C(F) \cap M(K_{m_{\ell}},F)}}||$ for each $\ell$. Since $u_{\gamma}^{\H(F)|_{C(F) \cap M(K_{m_{\ell}},F)}}$ is u.s.c., there exists $\mu_{m_{\ell}}$ in $C(F) \cap M(K_{m_{\ell}},F)$ such that $u_{\gamma}^{\H(F)|_{C(F) \cap M(K_{m_{\ell}},F)}}(\mu_{m_{\ell}}) = ||u_{\gamma}^{\H(F)|_{C(F) \cap M(K_{m_{\ell}},F)}}||$, for each $\ell$ in $\N$. Furthermore, $\{\mu_{m_{\ell}}\}_{\ell}$ tends to $x$ because $\{y_{m_{\ell}}\}_{\ell}$ tends to $x$. Then
\begin{align*} 
\limsup_{\ell} u_{\gamma}^{\H(F)|_{M(K_{m_{\ell}},F)}}(y_{m_{\ell}})+\htop(F|_{K_{m_{\ell}}}) & \leq \limsup_{\ell} u_{\gamma}^{\H(F)|_{C(F) \cap M(K_{m_{\ell}},F)}}(\mu_{m_{\ell}}) \\ 
 & \leq \limsup_{\ell} u_{\gamma}^{\H(F)|_{C(F)}}(\mu_{m_{\ell}})  \\
 & \leq u_{\gamma+1}^{\H(F)|_{C(F)}}(x). 
\end{align*}
It follows that 
\begin{equation} \label{AuxEqn1}
\limsup_{\substack{y_{m_{\ell}} \to x \\ y_{m_{\ell}} \in M_{m_{\ell}} }} u_{\gamma}^{\H(F)|_{M(K_{m_{\ell}},F)}}(y_{m_{\ell}})+\htop(F|_{K_{m_{\ell}}}) \leq u_{\gamma+1}^{\H(F)|_{C(F)}}(x).
\end{equation}
Now suppose that for each $m \geq 1$, $\al_0(F|_{K_m}) = 0$ and $\htop(F|_{K_m}) = h^F(\mu)$ for $\mu$ in $C(F) \cap M(K_m,F)$. Since $\al_0(F|_{K_m})=0$, we have that $u_{\gamma}^{\H(F)|_{M(K_m,F)}} \equiv 0$ for each $m \geq 1$. Let $\{y_{m_{\ell}}\}_{\ell}$ be a sequence tending to $x$ such that $y_{m_{\ell}}$ is in $M(K_{m_{\ell}},F)$ for each $\ell$. Let $\mu_{m_{\ell}}$ be in $C(F) \cap M(K_{m_{\ell}},F)$, for each $\ell$. Note that $\{\mu_{m_{\ell}}\}_{\ell}$ tends to $x$ because $\{y_{m_{\ell}}\}_{\ell}$ tends to $x$. By Proposition \ref{ESofFProp} (3), for each $k$, we may assume there exists a natural number $m_0$ such that for $m \geq m_0$, it holds that $h_{k}(\mu_m) = 0$, which implies that $\tau_{k}(\mu_m) = h^F(\mu_m)$. Then
\begin{align*}
\limsup_{\ell} & u_{\gamma}^{\H(F)|_{M(K_{m_{\ell}},F)}}(y_{m_{\ell}})+\htop(F|_{K_{m_{\ell}}}) = \limsup_{\ell} \htop(F|_{K_{m_{\ell}}}) \\ & =  \limsup_{\ell} h^F(\mu_{m_{\ell}})  =  \lim_k \limsup_{\ell} \tau_{k}(\mu_{m_{\ell}}) \leq u_1^{\H(F)|_{C(F)}}(x) \leq  u_{\gamma+1}^{\H(F)|_{C(F)}}(x).
\end{align*}
We have shown that in either case given by hypothesis (v), the third expression in the maximum in Equation (\ref{PenUltEqn1})-(\ref{PenUltEqn2}) is bounded above by $u_{\gamma+1}^{\H(F)|_{C(F)}}(x)$, as desired. Thus we have shown that $u_{\gamma+1}^{\H(F)|_M}(x) \leq u_{\gamma+1}^{\H(F)|_{C(F)}}(x)$, which finishes the successor case of our induction.

For the limit case, let $\gamma$ be a limit ordinal and suppose property (5) holds for all $\beta < \gamma$. Taking the limit supremum over the three sets $C(F)$, $M(K_0,F) \setminus C(F)$, and $\cup_{m \geq 1} M(K_m,F)$ in the definition of $u_{\gamma}^{\H(F)|_M}(x)$, we obtain
\begin{equation} \label{UltEqn}
u_{\gamma}^{\H(F)|_M}(x) \leq \max\Bigl( u_{\gamma}^{\H(F)|_{C(F)}}(x), \; u_{\gamma}^{\H(F)|_{M(K_0,F)}}(x), \; \limsup_{\substack{y_{m_{\ell}} \to x \\ y_{m_{\ell}} \in M_{m_{\ell}}}} u_{\gamma}^{\H(F)|_{M(K_{m_{\ell}},F)}}(y_{m_{\ell}}) \Bigr).
\end{equation}
By the same arguments as in the successor case, we bound the three expressions in the maximum in Equation (\ref{UltEqn}) from above by $u_{\gamma}^{\H(F)|_{C(F)}}(x)$, which shows that $u_{\gamma}^{\H(F)|_M}(x) \leq u_{\gamma}^{\H(F)|_{C(F)}}(x)$. This finishes our induction, and thus we have verified property (5).
\end{proof}

%%%%%%%%%%%%%%%%%%%%%%%%%%%%%%%%% Transfinite Sequence Facts %%%%%%%%%%%%%%%%%%%%%%%%%%%%%%%%%%%%%%%%%

\begin{lem} \label{HisHarmLemma}
Suppose $(X,F)$ is a topological dynamical system with entropy structure $\H(F)$. Suppose there exist closed sets $C$ and $M$ in $M(X,F)$ such that $C \subset M_{\erg}(X,F) \subset M$ and for all $x$ in $M$ and all ordinals $\gamma$,
\begin{equation*}
u_{\gamma}^{\H(F)|_M}(x) \leq \int_{C} u_{\gamma}^{\H(F)|_{C}} \; d\P_x.
\end{equation*}
Then for all $x$ in $M(X,F)$ and all ordinals $\gamma$,
\begin{equation*}
u_{\gamma}^{\H(F)}(x) = \int_{C} u_{\gamma}^{\H(F)|_{C}} \; d\P_x.
\end{equation*}
\end{lem}
\begin{proof}
Since $M$ is closed and contains $M_{\erg}(X,F)$, the Embedding Lemma (Lemma \ref{EmbeddingLemma}) implies that for all $x$ in $M(X,F)$ and all ordinals $\gamma$,
\begin{equation*}
u_{\gamma}^{\H(F)}(x) = \max_{\mu \in \Phi^{-1}(x)} \int_{M} u_{\gamma}^{\H(F)|_{M}} \; d\mu,
\end{equation*}
where $\Phi : \M(M) \to M(X,F)$ is the restriction of the barycenter map (which is onto since $M_{\erg}(X,F) \subset M$). By Fact \ref{TransSeqFacts} (2) and the fact that $\P_x \in \Phi^{-1}(x)$,
\begin{align*}
\max_{\mu \in \Phi^{-1}(x)} \int_{M} u_{\gamma}^{\H(F)|_{M}} \; d\mu & \geq \max_{\mu \in \Phi^{-1}(x)} \int_{C} u_{\gamma}^{\H(F)|_{C}} \; d\mu \\ & \geq  \int_{C} u_{\gamma}^{\H(F)|_{C}} \; d\P_x.
\end{align*}
For each ordinal $\gamma$, let $g_{\gamma} : M(X,F) \to [0,\infty)$ be defined by
\begin{equation*}
g_{\gamma}(x) = \left\{ \begin{array}{ll}
                 u_{\gamma}^{\H(F)|_{C}}(x), & \text{ if } x \in C \\
                 0, & \text{ otherwise.}
              \end{array}
       \right.
\end{equation*}
Note that since $C$ is closed and $u_{\gamma}^{\H(F)|_{C}}$ is u.s.c. and non-negative on $C$, we have that $g_{\gamma}$ is u.s.c. on $M(X,F)$. Also, $g_{\gamma}$ is convex for each $\gamma$ since it takes positive values only on extreme points (using that $C \subset M_{\erg}(X,F)$). Fact 2.5 in \cite{DM} (proved in \cite{BM}) states that the harmonic extension of a non-negative, convex, u.s.c. function is u.s.c. and of course harmonic. Applying this fact to $g_{\gamma}$, we obtain that the function $g^{\har}_{\gamma} : M(X,F) \to [0,\infty)$ defined by
\begin{equation*}
g^{\har}_{\gamma}(x) = \int g_{\gamma} \; d\P_x = \int_{C} u_{\gamma}^{\H(F)|_{C}} \; d\P_x
\end{equation*}
is harmonic and u.s.c. Then for any $\mu$ in $\Phi^{-1}(x)$, since $g^{\har}_{\gamma}$ is harmonic and $\mu$ is supported on $M$, we have that
\begin{equation*}
g_{\gamma}^{\har}(x) = g_{\gamma}^{\har}(\bary(\mu)) = \int_M g_{\gamma}^{\har} \; d\mu.
\end{equation*}
By hypothesis, we have
\begin{equation*}
u_{\gamma}^{\H(F)|_M}(x) \leq \int_{C} u_{\gamma}^{\H(F)|_{C}} \; d\P_x.
\end{equation*}
Combining all of these facts, we see that for $x$ in $M(X,F)$,
\begin{align*}
\int_{C} u_{\gamma}^{\H(F)|_{C}} \; d\P_x & \leq u_{\gamma}^{\H(F)}(x) \\
&  = \max_{\mu \in \Phi^{-1}(x)} \int_{M} u_{\gamma}^{\H(F)|_{M}} \; d\mu \\
& \leq \max_{\mu \in \Phi^{-1}(x)} \int_{M} \int_{C} u_{\gamma}^{\H(F)|_{C}} \; d\P_{\tau} \; d\mu(\tau) \\
& = \max_{\mu \in \Phi^{-1}(x)} \int_{M} g_{\gamma}^{\har} \; d\mu \\
& = g_{\gamma}^{\har}(x) \\
& = \int_{C} u_{\gamma}^{\H(F)|_{C}} \; d\P_x.
\end{align*}
Thus the above inequalities are actually equalities, and we have proved the lemma.
\end{proof}

%%%%%%%%%%%%%%%%%%%%%%%%%%%%%%%%%%%%%%%%%%%%%%%%%%%%%%%%%%%%%%%%%%%%%%%%%%%%%%%%%%%%%%%%%%%%%%%%
\section{Computation of some transfinite sequences} \label{TwoLemmas}

We will be interested in the following subsets of $\Cl_d$.
\begin{defn} \label{classesDefn}
Let $\al$ be a countable ordinal and $a \geq 0$. Let $\cS(\al,d,a)$ be the class of functions $f$ in $\Cl_d$ such that there exists a sequence of periodic orbits $\{\theta_m\}_{m \in \N}$ of $f$ such that the following conditions are satisfied, where $C(f) = \cap_{N=1}^{\infty} \overline{ \cup_{m \geq N} \{\mu_{\theta_m}\} }$:
\begin{enumerate}
 \item $f$ is ready for operation on $ \cup_m \theta_m$;
 \item for every $\mu$ in $C(f)$, $\mu$ is totally ergodic;
 \item if $\al =0$, then $C(f) = \{\nu\}$, where $\nu$ is the unique measure of maximal entropy for $f$
\item for all ordinals $\gamma$ and all points $x$ in $M(\D,f)$,
\begin{equation*}
u_{\gamma}^{\H(f)}(x) = \int\limits_{C(f)} u_{\gamma}^{\H(f)|_{C(f)}} \; d\P_x;
\end{equation*}
 \item $\al_0(f) = \al$.
 \item $||u_{\al}^{\H(f)}|| = a$.
\end{enumerate}
Also, let $\cS(\al,d) = \cup_{a \geq 0} \, \cS(\al,d,a)$.
\end{defn}

\begin{notation}
If $\{\theta_m\}_m$ is a sequence of periodic orbits for $f$ satisfying the conditions in Definition \ref{classesDefn} for $f$, then we write that $f$ is in $\cS(\al,d,a)$ with $\{\theta_m\}_m$.
\end{notation}

\begin{rmk}
For some pairs $\al$ and $a \geq 0$, the set $\cS(\al,d,a)$ is trivially empty. Indeed, if $\al = 0$ and $a >0$, then $\cS(\al,d,a)$ is empty. Also, if $\al >0$ and $a=0$, then $\cS(\al,d,a)$ is empty. On the other hand, in the course of proving Theorem \ref{mainThm}, we will show that for every countable ordinal $\al >0$, and every $a>0$, the set $\cS(\al,d,a)$ is non-empty.
\end{rmk}

\begin{lem} \label{alphaIsOneLemma}
Let $p$ be a non-negative integer and $a >0$. Suppose $f$, $\{\chi_m\}_m$, $\{\xi_m\}_m$, and $\{N_m\}_m$ satisfy the following conditions:
\begin{itemize}
 \item $f$ is in $\cS(p,d, \frac{a p}{p+1})$ with $\{\theta_m\}_m$;
 \item $||u_{\ell}^{\H(f)}||=\frac{a \ell}{p+1}$ for $\ell = 1, \dots, p$; %$\htop(f) = \max(\log(3), \frac{a}{p+1})$ and
 \item for each $m$, $N_m$ and $\xi_m$ are natural numbers and $1 \leq \xi_m \leq |\theta_m|$;
 \item for each $m$, $\chi_m$ is in $\cS(0,d)$ with $\{\theta^m_{\ell}\}_{\ell}$ and $\htop(\chi_m)= \log(N_m)$.
 \item the sequence $\{\frac{\xi_m}{|\theta_m|} \log(N_m)\}_m$ is increasing to $\frac{a}{p+1}$;
\end{itemize}
Then for any $F$ in $\Bl( f, \{\theta_m\}_m, \{\chi_m\}_m, \{\theta_{\ell}^m\}_{m,\ell}, \{\xi_m\}_m )$, $F$ is in $\cS(p+1,d,a)$, $\htop(F) = \max \bigl(\htop(f),\sup_m \frac{\xi_m}{|\theta_m|}\htop(\chi_m)\bigr)$, and $||u_{k}^{\H(F)}|| = \frac{a k}{p+1}$ for each $k$ in the set $\{1, \dots, p+1\}$.
% \item $\htop(F) = \max( \log(3), \frac{a}{p+1})$;
\end{lem}
\begin{proof} Let $f$, $\{\chi_m\}_m$, $\{\xi_m\}_m$, and $\{N_m\}_m$ be as above. By Proposition \ref{BlowAndSewProp}, there exists $F$ in $\Bl( f, \{\theta_m\}_m, \{\chi_m\}_m, \{\theta_{\ell}^m\}_{m,\ell}, \{\xi_m\}_m )$ with $\pi$, $\{K_i\}_{i=0}^{\infty}$, and $\{\phi_m\}_m$ as in Proposition \ref{BlowAndSewProp}. Then $F$ is in $\Cl_d$ and $F$ is ready for operation on the set $S= \cup_{m,\ell} \, \phi_m( \theta_{\ell}^m \times \{ 0, \dots, |\theta_m|-1\})$. Let $\Theta_k$ be an enumeration of the periodic orbits in $S$. Let
\begin{itemize}
 \item $C(F) = \cap_{n=1}^{\infty} \overline{\cup_{k \geq n} \{\mu_{\Theta_k}\}}$;
 \item $C(f) = \cap_{n=1}^{\infty} \overline{ \cup_{m \geq n} \{ \mu_{\theta_m}\} }$;
 \item $C(\chi_m) = \cap_{n=1}^{\infty} \overline{ \cup_{\ell \geq n} \{\mu_{\theta_{\ell}^m}\}}$;
 \item for each $i \geq 0$, $C(K_i,F) = C(F) \cap M(K_i,F)$.
\end{itemize}
To prove the lemma, we show the following:
\renewcommand{\labelenumi}{(\Alph{enumi})}
\begin{enumerate}
\item $\htop(F) = \max \bigl(\htop(f),\sup_m \frac{\xi_m}{|\theta_m|}\htop(\chi_m)\bigr)$;
\item for each $\mu$ in $C(F)$, $\mu$ is totally ergodic;
\item for each $x$ in $M(\D,F)$ and each ordinal $\gamma$,
\begin{equation*}
u_{\gamma}^{\H(F)}(x) = \int_{C(F)} u_{\gamma}^{\H(F)|_{C(F)}} \; d\P_x;
\end{equation*}
\item $\al_0(\H(F))= p+1$;
\item $||u_{k}^{\H(F)}|| = \frac{a \ell}{p+1}$ for $\ell = 1, \dots, p+1$.
\end{enumerate}
Corollary \ref{HtopCor} gives (A). Lemma \ref{totallyErgLemma} (1) implies (B). Lemma \ref{totallyErgLemma} (5) and Lemma \ref{HisHarmLemma} together imply (C).

Property (C) implies that $\al_0(\H(F)) \leq \al_0(\H(F)|_{C(F)}) $ and also that $||u_{k}^{\H(F)}|| = ||u_{k}^{\H(F)|_{C(F)}}|| $. Since $C(F) \subset M_{\erg}(\D,F)$, the measure $\P_x$ is just the point mass at $x$, for all $x$ in $C(F)$. With this fact, (C) implies that $u_{\gamma}^{\H(F)}(x) = u_{\gamma}^{\H(F)|_{C(F)}}(x)$ for all $x$ in $C(F)$. It follows that $\al_0(\H(F)) \geq \al_0(\H(F)|_{C(F)})$, and we conclude that in fact $\al_0(\H(F)) = \al_0(\H(F)|_{C(F)})$. We now observe that properties (D) and (E) will be satisfied once we show that $\al(\H(F)|_{C(F)}) = p+1$ and $||u_{\ell}^{\H(F)|_{C(F)}}|| = \frac{a \ell}{p+1}$ for $\ell = 1, \dots, p+1$. Let us prove these two facts by computing the transfinite sequence for $\H(F)|_{C(F)}$.

Note that for $m \geq 1$, $C(K_m,F)$ is open in $C(F)$ (by Lemma \ref{totallyErgLemma}). Then Fact \ref{TransSeqFacts} (1) and Lemma \ref{TowerLemmaThree} give that for all $x$ in $C(K_m,F)$,
\begin{align*}
u_{\gamma}^{\H(F)|_{C(F)}}(x) & = u_{\gamma}^{\H(F)|_{C(K_m,F)}}(x) \\
 & =  \frac{\xi_m}{|\theta_m|}u_{\gamma}^{\H(\chi_m)|_{C(\chi_m)}}( \psi_m( (\phi_m^{-1})(x) ) ).
\end{align*}
By the hypothesis that $\chi_m$ is in $\cS(0,d)$, $u_{\gamma}^{\H(\chi_m)} \equiv 0$ for all ordinals $\gamma$, and thus $u_{\gamma}^{\H(F)|_{C(K_m,F)}}(x) = 0$ for all $x$ in $C(F) \cap M(K_m,F)$.

For $x$ in $C(K_0,F)$, we have that for each $k$,
\begin{align*}
\limsup_{y \to x} \tau_k(y) & = \max\Bigl( \limsup_{\substack{y \to x \\ y \in C(K_0,F)}} \tau_k(y), \limsup_{\substack{y \to x \\ y \in C(F) \setminus C(K_0,F)}} \tau_k(y) \Bigr) \\
& \leq \max\Bigl( \limsup_{\substack{y \to x \\ y \in C(K_0,F)}} \tau_k(y), \, \limsup_m \frac{\xi_m}{|\theta_m|} \htop(\chi_m) \Bigr).
\end{align*}
Letting $k$ tend to infinity and using Lemma \ref{PrincExtLemma} (applied to $F|_{K_0}$, which is a principal extension of $f$ with factor map $\pi$) gives that
\begin{align} \label{U1Eqn}
u_1^{\H(F)|_{C(F)}}(x) & \leq \max\Bigl( u_1^{\H(F)|_{C(K_0,F)}}(x), \limsup_m \frac{\xi_m}{|\theta_m|} \htop(\chi_m)  \Bigr) \\ \label{U2Eqn}
& = \max\Bigl( u_1^{\H(f)}(\pi(x)), \limsup_m \frac{\xi_m}{|\theta_m|} \htop(\chi_m)  \Bigr).
\end{align}
By hypothesis, $||u_1^{\H(f)}||=\frac{a}{p+1}$ and $\lim_m \frac{\xi_m}{|\theta_m|} \htop(\chi_m) = \frac{a}{p+1}$. Then by Equations (\ref{U1Eqn}) and (\ref{U2Eqn}), we obtain $u_1^{\H(F)|_{C(F)}}(x) \leq \frac{a}{p+1}$. Since $x$ is in $C(K_0,F)$, there exist periodic orbits $\theta_{m_k}$ such that the sequence $\mu_{\theta_{m_k}}$ converges to $\pi(x)$. Let $\mu_{m_k}$ be the measure of maximal entropy for $F|_{K_{m_k}}$, which exists by the fact that $\chi_{m_k}$ is in $\cS(0,d)$ (property (3) in Definition \ref{classesDefn}). Then $\{\mu_{m_k}\}_k$ converges to $x$, and by the upper semi-continuity of $u_1^{\H(F))|_{C(F)}}$ and Proposition \ref{ESofFProp} (3), we have that
\begin{align*}
u_1^{\H(F)|_{C(F)}}(x) & \geq \lim_{\ell} \limsup_k (h^F-h^F_{\ell})(\mu_{m_k}) = \lim_{\ell} \limsup_k h^F(\mu_{m_k}) \\ & = \limsup_k \frac{\xi_{m_k}}{|\theta_{m_k}|} \htop(\chi_{m_k}) = \frac{a}{p+1}.
\end{align*}
This argument shows that for all $x$ in $C(K_0,F)$, it holds that $u_1^{\H(F)|_{C(F)}}(x) = \frac{a}{p+1}$. Now we claim that by induction on $\ell$, $u_{\ell}^{\H(F)|_{C(F)}}(x) = \frac{a}{p+1} + u_{\ell-1}^{\H(F)|_{C(K_0,F)}}(x)$ for $x$ in $C(K_0,F)$. The claim holds for $\ell = 1$. Assuming it holds for a natural number $\ell$, we have for $x$ in $C(K_0,F)$,
\begin{align*}
& \limsup_{y \to x} (u_{\ell}^{\H(F)|_{C(F)}}+\tau_k)(y) = \\ & = \max\Bigl( \limsup_{\substack{y \to x \\ y \in C(K_0,F)}} (u_{\ell}^{\H(F)|_{C(F)}}+\tau_k)(y), \limsup_{\substack{y \to x \\ y \in C(F) \setminus C(K_0,F)}} (u_{\ell}^{\H(F)|_{C(F)}}+\tau_k)(y) \Bigr) \\
& = \max\Bigl( \limsup_{\substack{y \to x \\ y \in C(K_0,F)}} \bigl(\frac{a}{p+1} + u_{\ell-1}^{\H(F)|_{C(K_0,F)}}+\tau_k\bigr)(y), \limsup_{\substack{y \to x \\ y \in C(F) \setminus C(K_0,F)}} \tau_k(y) \Bigr),
\end{align*}
where the second equality follows from the induction hypothesis on $\ell$ and the fact that $u_{\ell}^{\H(F)|_{C(K_m,F)}} \equiv 0$ for $m\geq 1$. Letting $k$ tend to infinity gives that
\begin{align*}
u_{\ell+1}^{\H(F)|_{C(F)} } (x) & = \max\Bigl( \frac{a}{p+1} + u_{\ell}^{\H(F)|_{C(K_0,F)}}(x) , \frac{a}{p+1} \Bigr) \\
 & = \frac{a}{p+1}+ u_{\ell}^{\H(F)|_{C(K_0,F)}}(x).
\end{align*}
By Lemma \ref{PrincExtLemma}, we have $u_{\ell}^{\H(F)|_{C(K_0,F)}}(x) = u_{\ell}^{\H(f)}(\pi(x))$ for all $x$ in $C(K_0,F)$. Now the facts $\al_0(\H(F)|_{C(F)}) = p+1$ and $||u_{\ell}^{\H(F)|_{C(F)}}|| = \frac{a \ell}{p+1}$ for $\ell = 1, \dots, p+1$ follow from the hypotheses on $f$ (in particular, $\al_0(\H(f)) = p$ and $||u_{\ell}^{\H(f)}|| = \frac{a \ell}{p+1}$ for $\ell = 1, \dots, p$). This concludes the proof of the lemma.
\end{proof}

\begin{lem} \label{alphaGOneLemma}
Let $\beta = 0$ or $\beta = \o^{\beta_1}+ \dots +\o^{\beta_k}$, where $\beta_1 \geq \dots \geq \beta_k$. Let $\al >1$ be an irreducible ordinal such that $\al \geq \o^{\beta_1}$ if $\beta \neq 0$. Let $a>0$ and $b \geq 0$. Suppose
\begin{itemize}
 \item $\{\al_m\}_m$ is a non-decreasing sequence of ordinals whose limit is $\al$;
 \item $\{\delta_m\}_m$ is a strictly increasing sequence of ordinals whose limit is $\al$;
 \item  $\{a_m\}_m$ is a sequence of positive real numbers tending to infinity;
 \item $f$ is in $\cS(\beta,d,b)$ with $\{\theta_m\}_m$;
 \item $||u_{\al}^{\H(f)}|| \leq a$;

 \item for each $m$, $\xi_m$ satisfies $1 \leq \xi_m \leq |\theta_m|$, and the sequence $\{\frac{\xi_m}{|\theta_m|} a_m\}_m$ is increasing to $a$;
 \item $\chi_m$ is in $\cS(\al_m,d,a_m)$ with $\{\theta_{\ell}^m\}_{\ell}$;
 \item $\frac{\xi_m}{|\theta_m|} \htop(\chi_m)$ tends to $0$;
 \item $\frac{\xi_m}{|\theta_m|}||u_{\delta_m}^{\H(\chi_m)}||$ tends to $0$.
\end{itemize}
Then for any $F$ in $\Bl( f, \{\theta_m\}_m, \{\chi_m\}_m, \{\theta_{\ell}^m\}_{m,\ell}, \{\xi_m\}_m )$, $F$ is in $\cS(\al+\beta,d,a+b)$ and for any ordinal $\gamma$,
\begin{equation} \label{UgammaNormEqn}
||u_{\gamma}^{\H(F)}|| = \left\{ \begin{array}{ll}
          \max \Bigl( ||u_{\gamma}^{\H(f)}||,\; \sup_m \, \frac{\xi_m}{|\theta_m|} ||u_{\gamma}^{\H(\chi_m)}|| \Bigr), & \text{ if } \gamma < \al \\
          a + ||u_{\gamma_0}^{\H(f)}|| , & \text{ if } \gamma = \al + \gamma_0.
\end{array} \right.
\end{equation}
Furthermore, if $\beta = 0$, then for any $\delta < \al$ and $0 < \epsilon < a$, there exists $m_0$ such that for any $F$ in $\Bl( f, \{\theta_{m+m_0}\}_m, \{\chi_{m+m_0}\}_m, \{\theta_{\ell}^{m+m_0}\}_{m,\ell}, \{\xi_{m+m_0}\}_m )$, $||u_{\delta}^{\H(F)}|| \leq \epsilon$.
\end{lem}
\begin{proof}
Let $f$, $\{\theta_m\}_m$, $\{\chi_m\}_m$, $\{\theta_{\ell}^m\}_{m,\ell}$, and  $\{\xi_m\}_m$. By Proposition \ref{BlowAndSewProp}, there exists $F$ in $\Bl( f, \{\theta_m\}_m, \{\chi_m\}_m, \{\theta_{\ell}^m\}_{m,\ell}, \{\xi_m\}_m )$ with $\pi$, $\{K_i\}_{i=0}^{\infty}$, and $\{\phi_m\}_m$ as in Proposition \ref{BlowAndSewProp}. Then $F$ is in $\Cl_d$ and $F$ is ready for operation on the set $S = \cup_{m,\ell} \, \phi_m( \theta_{\ell}^m \times \{ 0, \dots, |\theta_m|-1\})$. Let $\Theta_k$ be an enumeration of the periodic orbits in $S$. Let
\begin{itemize}
 \item $C(F) = \cap_{n=1}^{\infty} \overline{\cup_{k \geq n} \{\mu_{\Theta_k}\}}$;
 \item $C(f) = \cap_{n=1}^{\infty} \overline{ \cup_{m \geq n} \{ \mu_{\theta_m}\} }$;
 \item $C(\chi_m) = \cap_{n=1}^{\infty} \overline{ \cup_{\ell \geq n} \{\mu_{\theta_{\ell}^m}\}}$;
 \item for each $i \geq 0$, $C(K_i,F) = C(F) \cap M(K_i,F)$.
\end{itemize}
To prove the lemma, we will show the following:
\renewcommand{\labelenumi}{(\Alph{enumi})}
\begin{enumerate}
\item $\htop(F) = \max \bigl(\htop(f),\sup_m \frac{\xi_m}{|\theta_m|}\htop(\chi_m)\bigr)$;
\item for each $\mu$ in $C(F)$, $\mu$ is totally ergodic;
\item for each $x$ in $M(\D,F)$ and each ordinal $\gamma$,
\begin{equation*}
u_{\gamma}^{\H(F)}(x) = \int_{C(F)} u_{\gamma}^{\H(F)|_{C(F)}} \; d\P_x;
\end{equation*}
\item $\al_0(\H(F))= \alpha + \beta$;
\item for any ordinal $\gamma$, Equation (\ref{UgammaNormEqn}) holds.
\item if $\beta = 0$, then for any $\delta < \al$ and $0 < \epsilon < a$, there exists $m_0$ such that for any $F$ in $\Bl( f, \{\theta_{m+m_0}\}_m, \{\chi_{m+m_0}\}_m, \{\theta_{\ell}^{m+m_0}\}_{m,\ell}, \{\xi_{m+m_0}\}_m )$, $||u_{\delta}^{\H(F)}|| \leq \epsilon$.
\end{enumerate}
Corollary \ref{HtopCor} gives (A). Lemma \ref{totallyErgLemma} (1) implies (B). Lemma \ref{totallyErgLemma} (5) and Lemma \ref{HisHarmLemma} together imply (C).

Suppose that $\beta = 0$ and that $\delta < \al$ and $0 < \epsilon < a$ are given. Choose $m_0$ such that for all $m \geq m_0$, $\delta_{m} > \delta$ and $\frac{\xi_m}{|\theta_m|}||u_{\delta_m}^{\H(\chi_m)}|| < \epsilon$ (such $m_0$ exists by the hypotheses that $\delta_m$ tends to $\al$ and $\frac{\xi_m}{|\theta_m|}||u_{\delta_m}^{\H(\chi_m)}||$ tends to $0$). Then property (F) follows from property (E). It remains to show properties (D) and (E).

Property (C) implies that $\al_0(\H(F)) \leq \al_0(\H(F)|_{C(F)}) $ and that $||u_{k}^{\H(F)}|| = ||u_{k}^{\H(F)|_{C(F)}}|| $. Since $C(F) \subset M_{\erg}(\D,F)$, the measure $\P_x$ is just the point mass at $x$, for all $x$ in $C(F)$. Combining this fact with property (C) implies that $u_{\gamma}^{\H(F)}(x) = u_{\gamma}^{\H(F)|_{C(F)}}(x)$ for all $x$ in $C(F)$. It follows that $\al_0(\H(F)) \geq \al_0(\H(F)|_{C(F)})$ and therefore that $\al_0(\H(F)) = \al_0(\H(F)|_{C(F)})$. We now observe that properties (D) and (E) will be satisfied if we show that $\al(\H(F)|_{C(F)}) = \al+\beta$ and for all ordinals $\gamma$, Equation (\ref{UgammaNormEqn}) holds with $\H(F)$ replaced by $\H(F)|_{C(F)}$. Below we prove these two facts by computing the transfinite sequence for $\H(F)|_{C(F)}$, which will complete the proof.

Note that for $m \geq 1$, the set $C(K_m,F)$ is open in $C(F)$ by Lemma \ref{totallyErgLemma}. Then Fact \ref{TransSeqFacts} (1) and Lemma \ref{TowerLemmaThree} give that for all $x$ in $C(K_m,F)$,
\begin{align} \label{TransSeqOnKm}
u_{\gamma}^{\H(F)|_{C(F)}}(x) & = u_{\gamma}^{\H(F)|_{C(K_m,F)}}(x) \\
 & =  \frac{\xi_m}{|\theta_m|}u_{\gamma}^{\H(\chi_m)|_{C(\chi_m)}}( \psi_m( (\phi_m^{-1})(x) ) ).
\end{align}

We show by transfinite induction that for $\gamma < \al$ and $x$ in $C(K_0,F)$, we have $u_{\gamma}^{\H(F)|_{C(F)}}(x) = u_{\gamma}^{\H(F)|_{C(K_0,F)}}(x)$. The statement is trivially true for $\gamma = 0$. Suppose it holds for $\gamma < \al$. Then by the inductive hypothesis and Equation (\ref{TransSeqOnKm}), for $x$ in $C(K_0,F)$,
\begin{align*}
\limsup_{  y \to x}  (u_{\gamma}^{\H(F)|_{C(F)}}+\tau_k)(y) = & \max\Bigl( \limsup_{\substack{y \to x \\ y \in C(K_0,F)}} (u_{\gamma}^{\H(F)|_{C(F)}}+\tau_k)(y), \\ & \quad \quad \quad \quad \limsup_{\substack{y \to x \\ y \in C(F) \setminus C(K_0,F)}} (u_{\gamma}^{\H(F)|_{C(F)}}+\tau_k)(y) \Bigr) \\
& \leq \max\Bigl( \limsup_{\substack{y \to x \\ y \in C(K_0,F)}} (u_{\gamma}^{\H(F)|_{C(K_0,F)}}+\tau_k)(y), \\ & \quad \quad \quad \quad \limsup_{m}  ||u_{\gamma}^{\H(F)|_{C(K_m,F)}}|| + \frac{\xi_m}{|\theta_m|} \htop(\chi_m) \Bigr).
\end{align*}
Note that there exists $m_0$ such that $\delta_m > \gamma$ for all $m \geq m_0$, which implies that $u_{\gamma}^{\H(F)|_{C(K_m,F)}} \leq u_{\delta_m}^{\H(F)|_{C(K_m,F)}}$ for all large $m$. Then letting $k$ tend to infinity, we obtain
\begin{align*}
u_{\gamma+1}^{\H(F)|_{C(F)}}(x) & \leq \max\Bigl( u_{\gamma+1}^{\H(F)|_{C(K_0,F)}}(x), \limsup_m ||u_{\delta_m}^{\H(F)|_{C(K_m,F)}}|| + \frac{\xi_m}{|\theta_m|} \htop(\chi_m) \Bigr) \\
& = \max\Bigl( u_{\gamma+1}^{\H(F)|_{C(K_0,F)}}(x), \limsup_m \frac{\xi_m}{|\theta_m|}||u_{\delta_m}^{\H(\chi_m)}|| + \frac{\xi_m}{|\theta_m|} \htop(\chi_m) \Bigr) \\
& = \max\Bigl( u_{\gamma+1}^{\H(F)|_{C(K_0,F)}}(x), 0 \Bigr) \\
& = u_{\gamma+1}^{\H(F)|_{C(K_0,F)}}(x),
\end{align*}
using the hypotheses that $\frac{\xi_m}{|\theta_m|} ||u_{\delta_m}^{\H(\chi_m)}||$ tends to $0$ and $\frac{\xi_m}{|\theta_m|} \htop(\chi_m)$ tends to $0$.

By Fact \ref{TransSeqFacts} (2), $u_{\gamma+1}^{\H(F)|_{C(K_0,F)}}(x) \leq u_{\gamma+1}^{\H(F)|_{C(F)}}(x)$, and thus we conclude that $u _{\gamma+1}^{\H(F)|_{C(F)}}(x) = u_{\gamma+1}^{\H(F)|_{C(K_0,F)}}(x)$, which finishes the inductive step for successors.

Now suppose that $u _{\beta}^{\H(F)|_{C(F)}}(x) = u_{\beta}^{\H(F)|_{C(K_0,F)}}(x)$ holds for all $x$ in $C(K_0,F)$ and all $\beta < \gamma$, where $\gamma$ is a limit ordinal such that $\gamma < \al$. Recall that for $m$ sufficiently large, $\delta_m > \gamma$. Then by the induction hypothesis, for each $x$ in $C(K_0,F)$,
\begin{align*}
& u_{\gamma}^{\H(F)|_{C(F)}}(x) =  \limsup_{  y \to x}  \sup_{\beta < \gamma} u_{\beta}^{\H(F)|_{C(F)}}(y) \\
& \quad = \max\Bigl( \limsup_{\substack{y \to x \\ y \in C(K_0,F)}} \sup_{\beta < \gamma} u_{\beta}^{\H(F)|_{C(F)}}(y), \limsup_{\substack{y \to x \\ y \in C(F) \setminus C(K_0,F)}} \sup_{\beta < \gamma} u_{\beta}^{\H(F)|_{C(F)}}(y) \Bigr) \\
& \quad = \max\Bigl( \limsup_{\substack{y \to x \\ y \in C(K_0,F)}} \sup_{\beta < \gamma} u_{\beta}^{\H(F)|_{C(K_0,F)}}(y), \limsup_{\substack{y \to x \\ y \in C(F) \setminus C(K_0,F)}} u_{\gamma}^{\H(F)|_{C(F)}}(y) \Bigr) \\
& \quad \leq \max\Bigl( u_{\gamma}^{\H(F)|_{C(K_0,F)}}(x), \limsup_{m} \frac{\xi_m}{|\theta_m|} ||u_{\delta_m}^{\H(\chi_m)}|| \Bigr) \\
& \quad = \max\Bigl( u_{\gamma}^{\H(F)|_{C(K_0,F)}}(x), 0 \Bigr) \\
& \quad = u_{\gamma}^{\H(F)|_{C(K_0,F)}}(x)
\end{align*}
By Fact \ref{TransSeqFacts} (2), $u_{\gamma}^{\H(F)|_{C(K_0,F)}}(x) \leq u_{\gamma}^{\H(F)|_{C(F)}}(x)$, and we conclude that in fact $u _{\gamma}^{\H(F)|_{C(F)}}(x) = u_{\gamma}^{\H(F)|_{C(K_0,F)}}(x)$, which finishes the inductive step for limit ordinals. We have shown that for all ordinals $\gamma < \al$, $u _{\gamma}^{\H(F)|_{C(F)}}(x) = u_{\gamma}^{\H(F)|_{C(K_0,F)}}(x)$ for all $x$ in $C(F)$. Now by Lemma \ref{PrincExtLemma}, for $x$ in $C(K_0,F)$ and $\gamma < \al$,
\begin{equation*}
u _{\gamma}^{\H(F)|_{C(F)}}(x) = u_{\gamma}^{\H(F)|_{C(K_0,F)}}(x) = u_{\gamma}^{\H(f)}(\pi(x)).
\end{equation*}

At this point we conclude based on the above facts that for $\gamma < \al$,
\begin{equation*}
||u_{\gamma}^{\H(F)|_{C(F)}}|| = \max\Bigl( ||u_{\gamma}^{\H(f)}||, \sup_m \frac{\xi_m}{|\theta_m|} ||u_{\delta_m}^{\H(\chi_m)}|| \Bigr).
\end{equation*}

Since $\al$ is irreducible and greater than $1$, $\al$ is a limit ordinal. Thus for any $x$ in $C(K_0,F)$,
\begin{align*}
u_{\al}^{\H(F)|_{C(F)}}(x)  & = \limsup_{y \to x} \sup_{\beta < \al} u_{\beta}^{\H(F)|_{C(F)}}(y) \\
& \leq \max\Bigl( \limsup_{\substack{y \to x \\ y \in C(K_0,F)}} \sup_{\beta < \al} u_{\beta}^{\H(F)|_{C(K_0,F)}}(y), \limsup_{m}  ||u_{\al}^{\H(F)|_{C(K_m,F)}}|| \Bigr) \\
& = \max\Bigl( \limsup_{\substack{y \to x \\ y \in C(K_0,F)}} \sup_{\beta < \al} u_{\beta}^{\H(F)|_{C(K_0,F)}}(y), \limsup_{m}  \frac{\xi_m}{|\theta_m|} a_m \Bigr) \\
& \leq \max\Bigl( u_{\al}^{\H(F)|_{C(K_0,F)}}(x), a \Bigr).
\end{align*}
By hypothesis, $||u_{\al}^{\H(F)|_{C(K_0,F)}}|| \leq a$, and thus we have that $u_{\al}^{\H(F)|_{C(F)}}(x) \leq a$. On the other hand, since $x$ is in $C(K_0,F)$, there exists a sequence of periodic orbits $\{\theta_{m_k}\}_k$ such that $\{\mu_{\theta_{m_k}}\}_k$ converges to $\pi(x)$. If $\{\mu_{m_k}\}_k$ is a sequence of measures such that $\mu_{m_k}$ is in $M(K_{m_k},F)$ for each $k$, then $\{\mu_{m_k}\}_k$ converges to $x$, and we have
\begin{equation*}
u_{\al}^{\H(F)|_{C(F)}}(x) \geq \limsup_k ||u_{\al}^{\H(F)|_{C(K_{m_k})}}|| = \limsup_k \frac{\xi_{m_k}}{|\theta_{m_k}|} ||u_{\al}^{\H(\chi_{m_k})}|| = a.
\end{equation*}
It follows that for each $x$ in $C(K_0,F)$, $u_{\al}^{\H(F)|_{C(F)}}(x)=a$.

We show by induction that for $\gamma \geq 0$ and $x$ in $C(K_0,F)$,
\begin{equation} \label{uAlphaEqn}
u_{\al+\gamma}^{\H(F)|_{C(F)}}(x) = a + u_{\gamma}^{\H(F)|_{C(K_0,F)}}(x).
\end{equation}
Note that Equation (\ref{uAlphaEqn}) holds for $\gamma=0$. Now suppose Equation (\ref{uAlphaEqn}) holds for some ordinal $\gamma$. Then for all $x$ in $C(K_0,F)$,
\begin{align*}
\limsup_{  y \to x}  (u_{\al + \gamma}^{\H(F)|_{C(F)}}+\tau_k)(y) & = \max\Bigl( \limsup_{\substack{y \to x \\ y \in C(K_0,F)}} (u_{\al + \gamma}^{\H(F)|_{C(F)}}+\tau_k)(y), \\ & \quad \quad \quad \quad \limsup_{\substack{y \to x \\ y \in C(F) \setminus C(K_0,F)}} (u_{\al + \gamma}^{\H(F)|_{C(F)}}+\tau_k)(y) \Bigr) \\
& = \max\Bigl( \limsup_{\substack{y \to x \\ y \in C(K_0,F)}} (a+ u_{\gamma}^{\H(F)|_{C(K_0,F)}}+\tau_k)(y), \\ & \quad \quad \quad \quad \limsup_{\substack{m \to \infty \\ y \in C(K_m,F)}} (u_{\al}^{\H(F)|_{C(K_m,F)}}+\tau_k)(y) \Bigr).
\end{align*}
Taking the limit as $k$ tends to infinity gives
\begin{align*}
u_{\al+\gamma+1}^{\H(F)|_{C(F)}}(x) & = \max\Bigl( a+ u_{\gamma+1}^{\H(F)|_{C(K_0,F)}}(x), a +  \limsup_{m} \frac{\xi_m}{|\theta_m|}\htop(\chi_m) \Bigr) \\
& = \max\Bigl( a+ u_{\gamma+1}^{\H(F)|_{C(K_0,F)}}(x), a + 0 \Bigr) \\
& = a+ u_{\gamma+1}^{\H(F)|_{C(K_0,F)}}(x).
\end{align*}
This completes the inductive step for successor ordinals. Now suppose Equation (\ref{uAlphaEqn}) holds for all $\gamma < \beta$, where $\beta$ is a limit ordinal. Then for all $x$ in $C(K_0,F)$,
\begin{align*}
u_{\al+\beta}^{\H(F)|_{C(F)}}(x) & = \limsup_{y \to x} \sup_{\gamma<\beta} u_{\al+\gamma}^{\H(F)|_{C(F)}}(y) \\
& = \max\Bigl( \limsup_{\substack{y \to x \\ y \in C(K_0,F)}} \sup_{\gamma < \beta} u_{\al+ \gamma}^{\H(F)|_{C(K_0,F)}}(y), \limsup_{m}  ||u_{\al+\beta}^{\H(F)|_{C(K_m,F)}}|| \Bigr) \\
& = \max\Bigl( a+  u_{\beta}^{\H(F)|_{C(K_0,F)}}(x), a \Bigr) \\
& = a+ u_{\beta}^{\H(F)|_{C(K_0,F)}}(x),
\end{align*}
which completes the inductive step for limit ordinals. Combining Equation (\ref{uAlphaEqn}) with Lemma \ref{PrincExtLemma}, we obtain that
\begin{equation}
u_{\al+\gamma}^{\H(F)|_{C(F)}}(x) = a + u_{\gamma}^{\H(f)}(\pi(x)).
\end{equation}
Then Equation (\ref{UgammaNormEqn}) follows immediately and the equality $\al_0(\H(F)|_{C(F)}) = \al + \beta$ follows from the fact that $\al_0(\H(f)) = \beta$. This concludes the proof of the lemma.
\end{proof}

\section{Constructions by transfinite induction} \label{InductionScheme}

The following lemma serves as a base case for the transfinite induction construction in this section. Recall that for any countable ordinal $\al$ and any real number $a \geq 0$, the set $\cS(\al,d,a)$ was defined in Definition \ref{classesDefn}.

\begin{lem} \label{SzeroNonempty}
For any odd natural number $N \geq 3$, there exists $f$ in $\cS(0,d,0)$ such that $\htop(f)=\log(N)$.
\end{lem}
\begin{proof}
In the case $d=1$, let $f$ be the linear $N$-tent map on $[0,1]$. In the case $d=2$, let $f$ be an adaptation of Smale's $N$-horseshoe map such that $f : \D \to \D$ is a homeomorphism and $f|_{\partial \D} = \Id$. In either case, we have that $f$ is a continuous surjection, $f|_{\partial \D} = \Id$, and $\htop(f) = \log(N) < \infty$, which implies that $f$ is in $C_d$. Recall that $f$ has a unique measure of maximal entropy, which we denote as $\mu$. Also, there exists a sequence $\{\mu_{\theta_m}\}_m$ of periodic measures tending to $\mu$ with $\cup_m \theta_m$ contained in $\interior(\D)$. Fix such a sequence. Let $Q = \cup_k f^{-k}(\theta_m)$. Since $f$ is $N$-to-one when $d = 1$ and $f$ is injective when $d = 2$, we have that $Q$ is countable. Since $f$ has at most finitely many critical points, we assume without loss of generality that $Q$ contains no critical points, and thus $Df_x$ is invertible and continuous at $x$ for all $x$ in $Q$. Furthermore, we have that if $d = 2$, then $\det Df_x >0$ for $x$ in $Q$. We have shown that $f$ is ready for operation on $\cup_m \theta_m$. Now let $C(f) = \cap_{n=1}^{\infty} \overline{\cup_{m \geq n} \{\theta_m\}}$. Since $\{\theta_m\}_m$ tends to $\mu$, we have that $C(f) = \{\mu\}$. Also note that $\mu$ is totally ergodic. Recall that $h$-expansiveness (Definition \ref{hExpansiveDef}) implies that any entropy structure $(h_k)$ converges uniformly to $h$, which is equivalent to $u_{\al} \equiv 0$ and $\al_0(f)=0$ (see \cite{BD,D}). Since $f$ is $h$-expansive, we have that $u_{\al} \equiv 0$ for all $\al$ and $\al_0(f) = 0$. Hence we have shown that $f$ is in $\cS(0,d,0)$.
\end{proof}

\begin{lem} \label{BaseLemma}
Let $c \geq \log(3)$. Then for any $p$ in $\N$ and $a>0$, there exists $F$ in $\cS(p,d,a)$ such that $\htop(F) \leq \max(c, \frac{a}{p})$ and $||u_{k}^{\H(F)}|| = \frac{a k}{p}$ for $k = 1, \dots, p$.
\end{lem}
\begin{proof}
The proof is by induction on $p$. Consider the case $p=1$. By Lemma \ref{SzeroNonempty}, there exists $f$ in $\cS(0,d,0)$ with $\{\theta_m\}_m$ and $\htop(f) = \log(3)$. Choose $N_m$ and $\xi_m$ such that $1 \leq \xi_m \leq |\theta_m|$, $N_m \geq 3$, $N_m$ is odd, and $\{ \frac{ \xi_m}{|\theta_m|} \log(N_m)\}_m$ increases to $a$. By Lemma \ref{SzeroNonempty}, there exists $\chi_m$ in $\cS(0,d,0)$ with $\{\theta_{\ell}^m\}_{\ell}$ and $\htop(\chi_m) = \log(N_m)$. Then Proposition \ref{BlowAndSewProp} implies that there exists a function $F$ in $\Bl(f, \{\theta_m\}_m, \{\chi_m\}_m, \{\theta_{\ell}^m\}_{m,\ell}, \{\xi_m\}_m)$. Lemma \ref{alphaIsOneLemma} implies that $F$ is in $\cS(1,d,a)$. Also, $\htop(F) = \max( \htop(f), \sup_m \frac{\xi_m}{|\theta_m|} \htop(\chi_m) ) \leq \max( c, a)$.

Now assume the lemma holds for some $p$. By the induction hypothesis, let $f$ be in $S(p,d,\frac{ap}{p+1})$ with $\{\theta_m\}_m$ such that $\htop(f) \leq \max(c, \frac{a}{p+1})$ and $||u_k^{\H(f)}|| = \frac{a k}{p+1}$ for $k = 1, \dots, p$. Choose $N_m$ and $\xi_m$ such that $1 \leq \xi_m \leq |\theta_m|$, $N_m \geq 3$, $N_m$ is odd, and $\{ \frac{ \xi_m}{|\theta_m|} \log(N_m)\}_m$ increases to $\frac{a}{p+1}$. By Lemma \ref{SzeroNonempty}, there exists $\chi_m$ in $\cS(0,d,0)$ with $\{\theta_{\ell}^m\}_{\ell}$ and $\htop(\chi_m) = \log(N_m)$. Then Proposition \ref{BlowAndSewProp} implies that there exists a function $F$ in $\Bl(f, \{\theta_m\}_m, \{\chi_m\}_m, \{\theta_{\ell}^m\}_{m,\ell}, \{\xi_m\}_m)$. Lemma \ref{alphaIsOneLemma} implies that $F$ is in $\cS(p+1,d,a)$. Also, $\htop(F) = \max( \htop(f), \sup_m \frac{\xi_m}{|\theta_m|} \htop(\chi_m) ) \leq \max( c, \frac{a}{p+1})$, and $||u_k^{\H(F)}|| = \frac{a k}{p+1}$ for $k = 1, \dots, p+1$.
\end{proof}

%%%%%%%%%%%%%%%%%%%%%%%%%%%%%% Lemma 1 : Powers %%%%%%%%%%%%%%%%%%%%%%%%%%%%%%%%%%%%%%%%%%%%%%%%%%%%%%%%%%%%%%%%%%
\begin{lem} \label{powersLemma}
Let $\al >1$ be a countable, irreducible ordinal. Let $C > 0$. Suppose that for any ordinal $\delta < \al$, and any real numbers $\epsilon$ and $a$ such that $0 < \epsilon < a$, there exists $f$ in $\cS(\al,d,a)$ such that $\htop(f) \leq c$ and
\begin{equation*}
||u_{\delta}^{\H(f)}|| \leq \epsilon.
\end{equation*}
Then for any $a>0$, and any natural number $p > 1$, there exists $F$ in $\cS(\al p, d,a)$ such that $\htop(F) \leq c$ and
\begin{equation*}
||u_{\al \ell}^{\H(F)}|| = \frac{\ell}{p}a, \text{  for } \ell = 1, \dots, p.
\end{equation*}
\end{lem}
\begin{proof}
The proof proceeds by induction on $p$. We suppose it holds for $p$ and show it holds for $p+1$.

Let $f$ be in $\cS(\al p, d, \frac{a p }{p+1})$ with $\{\theta_m\}_m$ and satisfying the inductive hypotheses for $p$. Choose sequences $\{\delta_m\}_m$, $\{\xi_m\}_m$, and $\{a_m\}_m$ such that
\begin{itemize}
 \item $\{\delta_m\}_m$ is an increasing sequence of ordinals whose limit is $\al$;
 \item  $\{a_m\}_m$ is a sequence of positive real numbers tending to infinity;
 \item for each $m$, $\xi_m$ satisfies $1 \leq \xi_m \leq |\theta_m|$, and the sequence $\{\frac{\xi_m}{|\theta_m|} a_m\}_m$ is increasing to $\frac{a}{p+1}$.
\end{itemize}
Applying the hypothesis of the lemma, for each $m$ in $\N$, there exists $\chi_m$ in $\cS(\al,d, a_m)$ with $\{\theta_{\ell}^m\}_{\ell}$ such that $\htop(\chi_m) \leq c$ and  $||u_{\delta_m}^{\H(\chi_m)}|| \leq \min(\frac{a}{p+1},\frac{1}{m})$. Note that since $a_m$ tends to infinity and $\lim_m \frac{\xi_m}{|\theta_m|} a_m = \frac{a}{p+1}$, the sequence  $\{\frac{\xi_m}{|\theta_m|}\}$ tends to $0$. It follows that the sequence $\{ \frac{\xi_m}{|\theta_m|} \htop(\chi_m) \}_m$ tends to $0$. We assume without loss of generality that $\sup_m \{ \frac{\xi_m}{|\theta_m|} \htop(\chi_m) \}_m \leq c$ (if this inequality is not satisfied, replace $\chi_m$ by $\chi_{m+m_0}$ for sufficiently large $m_0$). Also, the sequence $\{\frac{\xi_m}{|\theta_m|} ||u_{\delta_m}^{\H(\chi_m)}||\}_m$ tends to $0$. Now by Proposition \ref{BlowAndSewProp}, there exists $F$ in $\Bl ( f, \{\theta_m\}_m, \{\chi_m\}_m, \{\theta^m_{\ell}\}_{m,\ell}, \{\xi_m\}_m)$. We have that $\htop(F) = \max( \htop(f), \sup_m \frac{\xi_m}{|\theta_m} \htop(\chi_m) ) \leq c$. By Lemma \ref{alphaGOneLemma}, $F$ is in $\cS(\al+ \al p, d, \frac{a}{p+1}+\frac{a p}{p+1}) = \cS(\al(p+1),d,a)$ and
\begin{itemize}
\item for any $\gamma < \al$, $||u_{\gamma}^{\H(F)}|| = \max \bigl( ||u_{\gamma}^{\H(f)}||, \, \sup_m ||u_{\gamma}^{\H(\chi_m)}|| \bigr)$;
 \item for $\gamma \geq 0$, $||u_{\al+\gamma}^{\H(F)}|| = \frac{a}{p+1}+||u_{\gamma}^{\H(f)}||$.
\end{itemize}
Then $||u_{\al}^{\H(F)}|| = \frac{a}{p+1}$, and the inductive hypotheses on $f$ imply that $||u_{\al \ell}^{\H(F)}|| = \frac{a \ell}{p+1}$ for $\ell = 1, \dots, p+1$. Thus $F$ satisfies the induction hypotheses for $p+1$, and by induction the lemma holds for all $p$.

\end{proof}

%%%%%%%%%%%%%%%%%%%%%%%%%%%%%%%%%%    Lemma 2: alpha irr.    %%%%%%%%%%%%%%%%%%%%%%%%%%%%%%%%%%%%%%%
%\subsubsection{Irreducible Orders Lemma}

\begin{lem} \label{irrLemma}
Let $\al> 1$ be a countable, irreducible ordinal. Let $c \geq \log(3)$. Then for all ordinals $\delta < \al$ and all real numbers $\epsilon$ and $a$ such that $ 0 < \epsilon < a$, there exists $F$ in $\cS(\al,d, a) $ such that $\htop(F) \leq c$ and
\begin{equation*}
||u_{\delta}^{\H(F)}|| \leq \epsilon.
\end{equation*}
\end{lem}
\begin{proof}
The proof is by transfinite induction on the irreducible ordinals $\al >1$. For notation, we let $\al = \o^{\beta}$, and use transfinite induction on $\beta \geq 1$.

\begin{case}[$\beta=1$] Let $f$ be in $\cS(0,d)$ with $\{\theta_m\}_m$ and $\htop(f)=\log(3)$ (such a map $f$ exists by Lemma \ref{SzeroNonempty}). Let $a$, $\epsilon$, and $\delta$ be as in the statement of the lemma. Choose sequences $\{a_m\}_m$ and $\{\xi_m\}_m$ such that
\begin{itemize}
 \item $\{a_m\}_m$ tends to infinity;
 \item $\xi_m$ is a natural number such that $1 \leq \xi_m \leq |\theta_m|$;
 \item the sequence $\{\frac{\xi_m}{|\theta_m|} a_m \}_m$ increases to $a$.
\end{itemize}
By Lemma \ref{BaseLemma}, there exists $\chi_m$ in $\cS(m,d,a_m)$ with $\{\theta^m_{\ell}\}_{\ell}$ and such that $\htop(\chi_m) \leq \max( c, \frac{a_m}{m})$ and $||u_{k}^{\H(\chi_m)}|| = \frac{a_m k}{m}$ for $k = 1, \dots, m$. Note that since $a_m$ tends to infinity and $\{\frac{\xi_m}{|\theta_m|} a_m\}_m$ increases to $a$, we have that $\{ \frac{\xi_m}{|\theta_m|} \htop(\chi_m) \}_m$ tends to $0$. Thus we assume without loss of generality that $\sup_m \frac{\xi_m}{|\theta_m|} \htop(\chi_m) \leq c$ (by replacing $\chi_m$ with $\chi_{m+m_0}$ for sufficiently large $m_0$ if necessary). Let $\delta_m = [\log(m)]$, the integer part of $\log(m)$. Then we obtain that $\{\frac{\xi_m}{|\theta_m|} ||u_{\delta_m}^{\H(\chi_m)}||\}_m$ tends to $0$ (since $\frac{\xi_m}{|\theta_m|} ||u_{\delta_m}^{\H(\chi_m)}|| = \frac{\xi_m a_m [\log(m)]}{|\theta_m| m} \leq \frac{a [\log(m)]}{m}$). By Proposition \ref{BlowAndSewProp}, there exists $F$ in $\Bl ( f, \{\theta_m\}_m, \{\chi_m\}_m, \{\theta^m_{\ell}\}_{m,\ell}, \{\xi_m\}_m)$. Then by Lemma \ref{alphaIsOneLemma}, $F$ is in $\cS(\o,d,a)$. Also, $\htop(F) = \max( \htop(f), \sup_m \frac{\xi_m}{|\theta_m|} \htop(\chi_m)) \leq c$. Also, the final statement in Lemma \ref{alphaIsOneLemma} gives that for $0< \epsilon < a$ and $\delta < \al$, there exists $m_0$ such that replacing $\chi_m$ with $\chi_{m+m_0}$ produces $F$ such that $||u_{\delta}^{\H(F)}|| \leq \epsilon$.
\end{case}

\begin{case}[successor ordinal]
Now suppose the lemma holds for the irreducible ordinal $\o^{\beta}$. We show that it also holds for $\o^{\beta+1}$. Let $f$ be in $\cS(0,d)$ with $\{\theta_m\}_m$ and $\htop(f) = \log(3)$. Choose sequences $\{\al_m\}_m$, $\{\delta_m\}$, $\{a_m\}_m$ and $\{\xi_m\}_m$ such that
\begin{itemize}
\item $\al_m = \o^{\beta} m $;
 \item $\delta_m = \o^{\beta} [\log(m)]$;
 \item  $\{a_m\}_m$ is a sequence of positive real numbers tending to infinity;
 \item for each $m$, $\xi_m$ satisfies $1 \leq \xi_m \leq |\theta_m|$, and the sequence $\{\frac{\xi_m}{|\theta_m|} a_m\}_m$ is increasing to $a$.
\end{itemize}
The inductive hypotheses imply that the hypotheses in Lemma \ref{powersLemma} are satisfied for $\o^{\beta}$. Applying Lemma \ref{powersLemma} for each $m$ in $\N$, we obtain that there exists $\chi_m$ such that
\begin{itemize}
 \item $\chi_m$ is in $\cS(\o^{\beta} m , d , a_m)$ with $\{\theta_{\ell}^m\}_{\ell}$;
 \item $ \htop(\chi_m) \leq c$;
 \item $||u_{\delta_m}^{\H(\chi_m)}|| = \frac{ a_m [\log(m)]}{m} $.
\end{itemize}
Since $\{a_m\}_m$ tends to infinity and $\{\frac{\xi_m}{|\theta_m|} a_m\}_m$ tends to $a$, $\{\frac{\xi_m}{|\theta_m|}\}_m$ tends to $0$. Therefore $\{\frac{\xi_m}{|\theta_m|} \htop(\chi_m) \}_m$ tends to $0$. Also, we have that $\{\frac{\xi_m}{|\theta_m|} ||u_{\delta_m}^{\H(\chi_m)}||\}_m$ tends to $0$. We assume without loss of generality that $\sup_m \frac{\xi_m}{|\theta_m|} ||u_{\delta_m}^{\H(\chi_m)}|| \leq \epsilon$. Now let $\delta < \al$ and $0< \epsilon < a$ be arbitrary. There exists $m_0$ such that $\delta_m > \delta$ for all $m \geq m_0$. Also, there exists $m_1$ such that $\frac{a [\log(m)]}{m} < \epsilon$ for all $m \geq m_1$. Let $m_2 = \max(m_0,m_1)$. Replace $\chi_m$ by $\chi_{m+m_2}$. By Proposition \ref{BlowAndSewProp}, there exists $F$ in $\Bl ( f, \{\theta_m\}_m, \{\chi_m\}_m, \{\theta^{m}_{\ell}\}_{m,\ell}, \{\xi_m\}_m)$. We have that $\htop(F) = \max( \htop(f), \sup_m \frac{\xi_m}{|\theta_m|} \htop(\chi_m)) \leq c$. Then Lemma \ref{alphaGOneLemma} implies that $F$ is in $\cS(\o^{\beta+1},d,a)$ and
\begin{equation*}
||u_{\delta}^{\H(F)}|| = \max \Bigl( ||u_{\delta}^{\H(f)}||, \; \sup_m \; \frac{\xi_m}{|\theta_m|} ||u_{\delta}^{\H(\chi_m)}|| \Bigr) = \sup_m \; \frac{\xi_m}{|\theta_m|} ||u_{\delta}^{\H(\chi_m)}|| \leq \epsilon,
\end{equation*}
as desired.
\end{case}

\begin{case}[$\beta$ limit ordinal]
Now suppose the lemma holds for all irreducible ordinals $\o^{\gamma} < \o^{\beta}$, where $\beta$ is a limit ordinal. We show that it also holds for $\o^{\beta}$. Let $f$ be in $\cS(0,d)$ with $\{\theta_m\}_m$ and $\htop(f) \leq c$. Choose a sequence $\{a_m\}_m$ of positive real numbers tending to infinity and an increasing sequence of ordinals $\{\beta_m\}_m$ tending to $\beta$. The inductive hypothesis implies that for each $m$, there exists $\chi_m$ in $\cS(\o^{\beta_m},d,a_m)$ with $\{\theta_{\ell}^m\}_{\ell}$ such that $\htop(\chi_m) \leq c$ and  $||u_{\o^{\beta_{m-1}}}^{\H(\chi_m)}|| \leq \frac{1}{m}$. Now let $\delta < \o^{\beta}$ and $\epsilon >0$ be arbitrary. There exists $m_0$ such that $\delta_m > \delta$ for all $m \geq m_0$. Also, there exists $m_1$ such that $\frac{1}{m} < \epsilon$ for all $m \geq m_1$. Let $m_2 = \max(m_0,m_1)$. Then replace $\chi_m$ by $\chi_{m+m_2}$. By Proposition \ref{BlowAndSewProp}, there exists $F$ in $\Bl ( f, \{\theta_m\}_m, \{\chi_m\}_m, \{\theta^{m}_{\ell}\}_{m,\ell}, \{\xi_m\}_m)$. By Corollary \ref{HtopCor}, we have that $\htop(F) = \max( \htop(f), \sup_m \frac{\xi_m}{|\theta_m|} \htop(\chi_m)) \leq c$. Then Lemma \ref{alphaGOneLemma} implies that $F$ is in $\cS(\o^{\beta},d,a)$ and
\begin{equation*}
||u_{\delta}^{\H(F)}|| = \max \Bigl( ||u_{\delta}^{\H(f)}||, \; \sup_m \; \frac{\xi_m}{|\theta_m|} ||u_{\delta}^{\H(\chi_m)}|| \Bigr) = \sup_m \; \frac{\xi_m}{|\theta_m|} ||u_{\delta}^{\H(\chi_m)}|| \leq \epsilon,
\end{equation*}
as desired.
\end{case}
\end{proof}

%%%%%%%%%%%%%%%%%%%%%%%%%%%%%%%%%%%%%%%   Conclusion of proof of mainIntThm    %%%%%%%%%%%%%%%%%%%%%%%%%%%%%5
\section{Proofs of the main results} \label{mainResults}

\begin{thm} \label{mainDiscThm}
Let $d$ be in $\{1,2\}$. For every countable ordinal $\al>0$ and any $a>0$, there is a map $F$ in $\cS(\al,d,a)$.
\end{thm}
\begin{proof}
Let $\al = \o^{\beta_1}+ \dots + \o^{\beta_n}$, with $\beta_1 \geq \dots \geq \beta_n$. We argue by induction on $n$. If $n=1$, then either Lemma \ref{BaseLemma} (if $\beta_1 = 0$) or Lemma \ref{irrLemma} (if $\beta_1 >0$) implies that there exists $F$ in $\cS(\al,d,a)$. Suppose the statement holds for $n$. We show that it holds for $n+1$. If $\beta_1 = 0$, then Lemma \ref{BaseLemma} implies that $F$ exists with the desired properties. Now suppose $\beta_1 >0$. Let $a_1 \geq a_0 >0$ with $a_1+a_0 = a$. By the induction hypothesis, there exists $f$ in $\cS(\o^{\beta_2}+\dots+\o^{\beta_n},d,a_0)$ with $\{\theta_m\}_m$. Choose sequences $\{a_m\}_m$, $\{\delta_m\}_m$, and $\{\xi_m\}_m$ such that
\begin{itemize}
\item $\{a_m\}_m$ is a sequence of positive numbers tending to infinity;
\item $\{\delta_m\}_m$ is an increasing sequence of ordinals tending to $\o^{\beta_1}$;
\item $1 \leq \xi_m \leq |\theta_m|$ and the sequence $\{\frac{\xi_m}{|\theta_m|} a_m\}_m$ increases to $a$.
\end{itemize}
Let $c \geq \log(3)$. Then for each $m$, Lemma \ref{irrLemma} implies that there exists $\chi_m$ in $\cS(\o^{\beta_1},d,a_m)$ with $\{\theta^m_{\ell}\}_{\ell}$ such that $\htop(\chi_m) \leq c$ and $||u_{\delta_m}^{\H(\chi_m)}|| \leq \frac{1}{m}$. Note that $\{\frac{\xi_m}{|\theta_m|} \htop(\chi_m)\}$ tends to $0$ with these choices of parameters. By Proposition \ref{BlowAndSewProp}, there exists $F$ in $\Bl ( f, \{\theta_m\}_m, \{\chi'_m\}_m, \{\theta^{m'}_{\ell}\}_{m,\ell}, \{\xi_m\}_m)$. By Lemma \ref{alphaGOneLemma}, $F$ is in $\cS(\o^{\beta_1}+\o^{\beta_2}+\dots+ \o^{\beta_n},d,a_1+a_0) = \cS(\al,d,a)$, which completes the induction and the proof.
\end{proof}

%The proof follows the outline of the proof of Corollary \ref{realizationCor}, using lemmas \ref{intervalPowerLemma} and \ref{intervalMainTech}.

\begin{cor} \label{higherDimCor}
Let $\al$ be a countable ordinal, let $a >0$ and let $d$ be in $\N$. Let $\D$ be the closed unit ball in $\R^d$. Then there exists a continuous surjection $f : \D \to \D$ such that $f|_{\partial \D} = \Id$, $\al_0(f) = \al$, and $||u_{\al}^{\H(f)}||=a$. If $d \geq 2$, then $f$ can be chosen to be a homeomorphism.
\end{cor}
\begin{proof}
If $d$ is $1$ or $2$, then Theorem \ref{mainDiscThm} implies that there exists $g$ in $\cS(\al,d,a)$, which satisfies the conclusion of the corollary. We remark that since $\D$ and $[-1,1]^d$ are homeomorphic, the statement of the theorem is equivalent to the same statement with $[-1,1]^d$ in place of $\D$. Thus we may consider all maps defined on $[-1,1]^d$ without loss of generality. The proof proceeds by induction on $d$. Suppose the corollary holds for some $d \geq 2$. Using this inductive hypothesis, choose a homeomorphism $g : [-1,1]^d \to [-1,1]^d$ such that $g|_{\partial [-1,1]^d} = \Id$, $\al_0(g) = \al$, and $||u_{\al}^{\H(g)}||=a$. Then there exists a homeomorphism $f : [-1,1]^{d+1} \to [-1,1]^{d+1}$ such that $f|_{\partial [-1,1]^{d+1}} = \Id$, $f(x,0) = (g(x),0)$ for $x$ in $[-1,1]^d$, and $\NW(f) \subset \{(x,t) \in [-1,1]^d \times [-1,1] : t = 0\} \cup \partial [-1,1]^{d+1}$. Such a map $f$ may be constructed as follows. Let
\begin{equation*}
V = \{(x_1, \dots, x_d, t) \in [-1,1]^{d+1} : |x_i| \leq (1 - |t|) \text{ for } 1 \leq i \leq d \}. 
\end{equation*}
Also, define $T : [-1,1] \to [-1,1]$ by $T(t) = t + \frac{1}{10} \sin(\pi t)$. For $x$ in $\partial [-1,1]^{d+1}$, let $f(x) = x$. For $(x,t)$ in $V$ (where $x \in [-1,1]^d$) such that $|t|<1$, let
\begin{equation*}
f(x,t) = \Biggl( (1-|T(t)|) g\biggl(\frac{1}{(1-|t|)} x\biggr) , \, T(t) \Biggr).
\end{equation*}
We have defined $f$ on $V \cup \partial [-1,1]^{d+1}$. For any point $p$ in $[-1,1]^{d+1} \setminus (V \cup \partial [-1,1]^{d+1})$, let $\ell_p$ denote the line in $\R^{d+1}$ passing through $p$ and the origin. Let $p_1$ and $p_2$ be the points such that $\{p_1\} = \partial V \cap \ell_p$ and $\{p_2\} = \partial[-1,1]^{d+1} \cap \ell_p$. Then let $s$ in $[0,1]$ be such that $p = s p_1 + (1-s) p_2$. Now define $f(p) = s f(p_1) + (1-s) f(p_2)$. With this definition, $f$ is a homeomorphism of $[-1,1]^{d+1}$ (using that $g|_{\partial [-1,1]^d} = \Id$). Furthermore, we have that $f|_{\partial [-1,1]^{d+1}} = \Id$, $f(x,0) = (g(x),0)$ for $x$ in $[-1,1]^d$, and $\NW(f) \subset \{(x,t) \in [-1,1]^d \times [-1,1] : t = 0\} \cup \partial [-1,1]^{d+1}$. Then $\al_0(f) = \al_0(g) = \al$ and $||u_{\al}^{\H(f)}|| = ||u_{\al}^{\H(g)}|| = a$. In this way we have verified the inductive hypotheses for $d+1$, which finishes the proof of the corollary.
\end{proof}

The following theorem, which we view as our main result, answers a question of Todd Fisher.

\begin{thm} \label{mainThm}
 Let $\al$ be a countable ordinal and let $a>0$. Let $M$ be a compact manifold. Then there exists a continuous surjection $f : M \to M$ such that $\al_0(f)=\al$ and $||u_{\al}^{\H(f)}|| = a$. If $\dim(M) \geq 2$, then $f$ can be chosen to be a homeomorphism.
\end{thm}
\begin{proof}
Let $d = \dim(M)$, and let $\D$ be the closed unit ball in $\R^d$. By Corollary \ref{higherDimCor}, there exists a continuous onto map $g : \D \to \D$ such that $g|_{\partial \D} = \Id$, $\al_0(g)=\al$, $||u_{\al}^{\H(g)}||=a$, and $g$ is a homeomorphism if $d \geq 2$.  We define a map $G : \D \to \D$ as follows. Let $G|_{\overline{B(\0,\frac{1}{2})}} = A_{\frac{1}{2},\0} \circ g \circ A_{2,\0}$, where $A_{s,\mathbf{p}}$ is the affine map on $\R^d$ given by $A_{s,\mathbf{p}}(x) = sx+\mathbf{p}$ and $\0$ is the origin. Now parametrize the annulus $\{x \in \R^d : \frac{1}{2} \leq |x| \leq 1\}$ with polar coordinates $(r,\theta) \in [\frac{1}{2},1] \times S^1$. For $(r,\theta)$ in $[\frac{1}{2},1] \times S^1$, let $G(r,\theta) = (r + \frac{1}{10} \sin ( 2 \pi r), \theta)$. Now $G$ is in $\Cl_d$ and $\overline{B(\0,\frac{1}{2})}$ is an isolated set for $G$. Let $\phi : \D \rightarrow M$ be a homeomorphism onto its image (such a map exists since $M$ is a manifold). Define $f : M \to M$ as follows. For $x$ in $\phi(\D)$, let $f(x) = \phi( G ( \phi^{-1}(x) ) )$. For $x$ in $M \setminus \phi(\D)$, let $f(x)=x$. Then $NW(f) = \phi(\overline{B(\0,\frac{1}{2})}) \cup ( M \setminus \phi(\interior(\D)) )$. Further, $f$ is topologically conjugate to $G|_{\overline{B(\0,\frac{1}{2})}}$ on $\phi(\overline{B(\0,\frac{1}{2})})$, and $f$ is the identity on $M \setminus \phi( \interior(D))$. It follows that $\al_0(f) = \al_0(G) = \al_0(g) = \al$ and $||u_{\al}^{\H(f)}|| = ||u_{\al}^{\H(G)}|| = ||u_{\al}^{\H(g)}|| = a$.
\end{proof}

%%%%%%%%%%%%%%%%%%%%%%%%%%%%%%%%%%%%%%%%%%%%%%%%%%%%%%%%%%%%%%%%%%%%%%%%%%%%%%%%%%%%%%%%%%%%%%%%%%%%%%%%
%%%%%%%%%%%%%%%%%%%%%%%%%%%%%%%%%%%%%%%%%%%%%%%%%%%%%%%%%%%%%%%%%%%%%%%%%%%%%%%%%%%%%%%%%%%%%%%%%%%%%%%%

\end{document}